 \documentclass{amsart}
\usepackage{amsmath,amsthm,latexsym,amssymb}
\usepackage{eucal}
\usepackage{mathrsfs}
\usepackage[all]{xy}
\xyoption{2cell}

\numberwithin{equation}{section}
\newtheorem{thm}{Theorem}[section]
\newtheorem{lem}[thm]{Lemma}
\newtheorem{prop}[thm]{Proposition}
\newtheorem{cor}[thm]{Corollary}
\newtheorem{conj}[thm]{Conjecture}
\theoremstyle{definition}
\newtheorem{defn}[thm]{Definition}
\theoremstyle{remark}
\newtheorem{rmk}[thm]{Remark}

\newtheorem{ex}[thm]{Example}
\newtheorem{exs}[thm]{Examples}

\newtheorem{notn}[thm]{Notation}
\newtheorem{convention}[thm]{Convention}

\newtheorem{acknowledge}[thm]{Acknowledgement}

\newcommand\Om{\Omega}

\newcommand\vp{\varphi}
\newcommand \ve{\varepsilon}

\newcommand\si{s^{-1}}

\newcommand{\cat}{\mathbf}
\newcommand{\op}{\mathcal}

\newcommand{\ob}{\operatorname{Ob}}
\newcommand{\mor}{\operatorname{Mor}}

\newcommand\egal[2]{\overset {#1}{\underset {#2}\rightrightarrows }}
\newcommand\adjunct[2]{\overset {#1}{\underset {#2}\rightleftarrows }}
\newcommand\fib{\ar @{->>} [r]} 
 
\newcommand\cof{\;\ar@{ >->}[r]} 
\newcommand\ccof{\;\ar@{ >->}[rr]} 
\newcommand\we{\ar [r]^{\sim}}
\newdir{ >}{{}*!/-10pt/\dir{>}}
\newcommand\wefib{\fib ^{\sim}} 
\newcommand\wecof{\cof^{\sim}}

\newcommand\zz {\mathbb Z}

\begin{document}

\title[Homotopic Hopf-Galois extensions]{Homotopic Hopf-Galois extensions:\\ Foundations and examples}

\author{Kathryn Hess}

\address{Institut de g\'eom\'etrie, alg\`ebre et topologie (IGAT) \\
    \'Ecole Polytechnique F\'ed\'erale de Lausanne \\
    CH-1015 Lausanne \\
    Switzerland}
\email{kathryn.hess@epfl.ch}
    
\begin{abstract} Hopf-Galois extensions of rings generalize  Galois extensions, with the coaction of a Hopf algebra replacing the action of a group.  Galois extensions with respect to a group $G$ are the Hopf-Galois extensions with respect to the dual of the group algebra of $G$.  Rognes recently defined an analogous notion of Hopf-Galois extensions in the category of structured ring spectra, motivated by the fundamental example of the unit map from the sphere spectrum to $MU$. 

This article introduces a  theory of homotopic Hopf-Galois extensions in a monoidal category with compatible model category structure that generalizes the case of structured ring spectra.  In particular, we provide explicit examples of homotopic Hopf-Galois extensions in various categories of interest to topologists, showing that, for example, a principal fibration of simplicial monoids is a homotopic Hopf-Galois extension in the category of simplicial sets.  We also investigate the relation of homotopic Hopf-Galois extensions to descent.
\end{abstract}

\date{\today}

\maketitle

\tableofcontents

\section*{Introduction}

The goal of this paper is to lay the foundations of a theory of Hopf-Galois extensions in monoidal model categories, generalizing both the classical case of rings \cite{montgomery}, \cite{schauenburg}  and its extension to ``brave new rings,'' i.e., ring spectra \cite{rognes}.  We begin by recalling the classical notion.

\begin{defn}\label{defn:hg-classical}  Let $\Bbbk$ be a commutative ring, and let $B$ be a $\Bbbk$-algebra, endowed with an augmentation $\ve: B\to \Bbbk$.  Let $\varphi: B\to A$ be a  homomorphism of $\Bbbk$-algebras. Let $H$ be a bialgebra, considered as a $B$-algebra with trivial $B$-action, i.e., the action determined by the composite $B\xrightarrow \ve \Bbbk \xrightarrow \eta H$, where $\eta$ is the unit of $H$.   

The homomorphism $\varphi$ is an \emph{$H$-Hopf-Galois extension} if $A$ admits a right $H$-coaction $\rho:A\to A\otimes H$, which is a morphism of $B$-algebras such that 
\begin{enumerate}
\item the composite
$$A\underset B\otimes A \xrightarrow {A\underset B \otimes \rho} A\underset B\otimes A\otimes H \xrightarrow{\mu \otimes H} A\otimes H,$$
where $\mu$ denotes the multiplication map of $A$ as a $B$-algebra,
and 
\item the induced map
$$B\to A^{co\, H}:=A\underset H\square \Bbbk =\{ a\in A\mid \rho(a)=a\otimes 1\}$$
\end{enumerate}
are both isomorphisms.
\end{defn}

\begin{notn}  The composite in (1) is usually denoted $\beta: A\underset B\otimes A \to A\otimes H$ and called the \emph{Galois map}, while the induced map in (2) is usually denoted $i: B\to A^{co\, H}$.
\end{notn}

\begin{ex} Let $G$ be a group.  If $\varphi: B \to A$ is a $G$-Galois extension of commutative rings, then it is a $\operatorname{Hom} (\zz [G], \zz)$-Hopf-Galois extension.
\end{ex}

\begin{ex}  Let $\Bbbk$ be a commutative ring.  Let $H$ be a Hopf algebra over $\Bbbk$ that is flat as $\Bbbk$-module, and let $A$ be a flat $\Bbbk$-algebra.  Then the trivial extension $A\to A\otimes H: a \mapsto a\otimes 1$ is an $H$-Hopf-Galois extension.
\end{ex}

For further discussion of the classical theory of Hopf-Galois extensions, we refer the reader to the article \cite{montgomery} by Montgomery in these proceedings.

In his monograph on Galois extensions of structured ring spectra \cite{rognes}, Rognes observed that the unit map from the sphere spectrum $S$ to $MU$ was a $S[BU]$-Hopf-Galois extension in a homotopical sense, where
\begin{itemize}
\item the diagonal $\Delta:BU\to BU\times BU$ induces the comultiplication $S[BU]\to S[BU]\wedge S[BU]$;
\item the Thom diagonal $ MU\to MU\wedge BU_{+}$ gives rise to the coaction of $S[BU]$ on $MU$; and
\item $\beta:MU\wedge MU \xrightarrow \simeq MU\wedge S[BU]$ is the Thom equivalence.
\end{itemize}
This article is motivated by the desire to provide a general framework in which to study such homotopic Hopf-Galois extensions.

The generalization of Hopf-Galois extensions to categories with compatible mon\-oi\-dal and model structures (Definition \ref{defn:hg-homotopic}) proceeds essentially by asking that the maps $\beta $ and $i$ be weak equivalences rather than isomorphisms and by taking the homotopy coinvariants of the coaction of $H$, rather than ordinary coinvariants.  In fact we ``categorify'' condition (2) of Definition \ref{defn:hg-classical}, promoting it to a condition on homotopy categories of modules.  As we explain in Remark \ref{rmk:tautology}, we speculate that the ``correct'' definition of homotopic Hopf-Galois extensions may require categorification of condition (1) of Definition \ref{defn:hg-classical} as well.  For the purposes of this paper, we have chosen not to do so, but further experience with this notion may lead to the consensus that one should.

The key problem that we must solve before defining homotopic Hopf-Galois extensions is to determine how to compute the homotopy coinvariants of a coaction, in particular when taking multiplicative structure into account.  Our discussion of this problem forms the heart of this paper.  

We begin in section 1 by developing a framework for studying the homotopy theory of comodules.  In particular, we provide conditions under which a category of comodules in a monoidal model category admits a reasonable model structure.  In section 2 we explain how to define homotopy coinvariants of a coaction, in terms of the homotopy theory defined in section 1, and apply the theory to a number of specific categories.  We show in particular that there is a reasonable model category structure on the category of comodules over a fixed comonoid, when the underlying category is that of simplicial sets, simplicial monoids, chain complexes over a field or chain algebras over a field.  We then give explicit formulas for the homotopy coinvariants of a coaction in each of these cases.

The definition of homotopic Hopf-Galois extensions is formulated in section 3.  We show that trivial extensions are indeed homotopic Hopf-Galois extensions under reasonable conditions and provide examples of homotopic Hopf-Galois extensions in the categories of simplicial monoids and of chain algebras.  Finally, in section 4 we initiate a study of the theory of homotopic Hopf-Galois extensions, exploring their relation to notions of (homotopic) faithful flatness and descent, within the general framework of the homotopy theory of comonoids over co-rings. 

Essential definitions and terminology concerning model categories are recalled in the appendix, where we also prove useful existence results (Theorem \ref{thm:left-ind} and Corollary \ref{cor:postnikov}) for model structures induced from right to left across adjunctions. Our discussion of the homotopy theory of comodules and of comodules over co-rings is based on these existence results.

In a follow-up to this paper, the theory of homotopic Hopf-Galois extensions, including the behavior of extensions under cobase change, extensions of commutative monoids and the proof of one direction of the Hopf-Galois correspondence, will be developed in greater depth.  Further examples, such as the categories of rational, commutative cochain algebras and of symmetric spectra, will also be treated.

\begin{notn} Let $\cat M$ be a small category, and let $A,B\in \ob \cat M$.  In these notes, the set of morphisms from $A$ to $B$ is denoted $\cat M(A,B)$.  The identity morphisms on an object $A$ will often be denoted $A$ as well.
\end{notn}

\begin {acknowledge}This project began with the masters thesis of C\'edric Bujard \cite{bujard}, supervised by the author, in which  a theory of homotopic Hopf-Galois extensions was first sketched.  The formulation of the theory presented in this paper has its roots in Bujard's thesis. 

The author would like to thank Bill Dwyer for an enlightening discussion of the appropriate definition of homotopy coinvariants and Susan Montgomery for suggesting Schauenburg's survey article \cite{schauenburg}.  The author also extends her gratitude to Andy Baker and Birgit Richter for having organized a fantastic workshop at the Banff International Research Station.  Finally, the author greatly appreciated the constructive comments of referee.
\end{acknowledge}

\section {Homotopy theory of comodules}\label{sec:cats}

We recall the definition of comonoids and of their comodules in a monoidal category.  We then provide conditions under which the category of comodules over a fixed comonoid  admits a reasonable model category structure, inherited from that of the underlying category.

\subsection{Comonoids and their comodules}

Throughout this section $(\cat M, \otimes , I)$ denotes any monoidal category. 

The following definition dualizes the familiar notion of monoids in a monoidal category.

\begin{defn}  A \emph{comonoid} in $\cat M$ is an object $C$ in $\cat M$, together with two morphisms in $\cat M$: a comultiplication map $\Delta:C\to C\otimes C$ and a counit map $\ve :C\to I$ such that $\Delta$ is coassociative and counital, i.e., the diagrams
$$\xymatrix{
C\ar [d]_{\Delta }\ar [rr]^{\Delta}&&C\otimes C\ar[d]_{\Delta\otimes C}&&C\ar[r]^\Delta\ar[d]_{\cong}&C\otimes C\ar[dl]^{C\otimes \ve}\ar[dr]_{\ve\otimes C}&C\ar[l]_{\Delta}\ar[d]^\cong\\
C\otimes C\ar [rr]^{C\otimes \Delta}&&C&&C\otimes I&&I\otimes C}$$
must commute, where the isomorphisms are the natural isomorphisms of the monoidal structure on $\cat M$.   

A comonoid $(C,\Delta , \ve)$ that is endowed with a comonoid map $\eta: I\to C$, where the comultiplication on $I$ is the natural isomorphism $I\xrightarrow \cong I\otimes I$, is said to be \emph{coaugmented}.

Let $(C,\Delta, \ve)$ and $(C',\Delta',\ve')$ be comonoids in a monoidal category $(\cat M,\otimes ,I)$. A \emph{morphism of comonoids} from $(C,\Delta, \ve)$ to $(C',\Delta',\ve')$ is a morphism $f\in \cat M(C,C')$ such that the diagrams
$$\xymatrix{
C\ar [r]^f\ar [d]_{\Delta}&C'\ar [d]_{\Delta'}&&C\ar[rr]^f\ar [dr]_{\ve}&&C''\ar [dl]^{\ve'}\\
C\otimes C\ar [r]^{f\otimes f}&C'\otimes C'&&&I
}$$ 
commute.
\end{defn}  

\begin{notn} We often abuse terminology slightly and refer to a (co)monoid simply by its underlying object in the category $\cat M$, just as we sometimes write only the underlying category when naming a monoidal category.
\end{notn}

\begin{rmk}  If $\cat M$ is a symmetric monoidal category, the category $\cat {Alg}$ of monoids in $\cat M$ is itself a monoidal category, where the multiplication on a tensor product of monoids $(A,\mu)$ and $(A',\mu')$ is given by the composite
$$(A\otimes A')\otimes (A\otimes A')\cong (A\otimes A)\otimes (A'\otimes A') \xrightarrow {\mu\otimes \mu'} A\otimes A'.$$
A comonoid in $\cat {Alg}$ is called a \emph{bimonoid} and consists of an object $H$ in $\cat M$, together with a multiplication $\mu: H\otimes H\to H$, a comultiplication $\Delta:H\to H\otimes H$, a unit $\eta:I\to H$ and a counit $\ve: H\to I$, which are appropriately compatible. Note that any bimonoid is automatically coaugmented as a comonoid, via the unit $\eta$.
\end{rmk}

\begin{defn} Let $(C,\Delta, \ve)$ be a comonoid in a monoidal category $(\cat M, \otimes, I)$.  A \emph{right $C$-comodule} in $\cat M$ is an object $M$ in $\cat M$ together with a morphism $\rho:M\to M\otimes C$ in $\cat M$, called the \emph{coaction map}, such that the diagrams
$$\xymatrix{
M\ar [d]_{\rho}\ar [rr]^{\rho}&&M\otimes C\ar[d]_{\rho\otimes C}&&M\ar [r]^{\rho}\ar [d]_\cong&M\otimes C\ar [dl]^{M\otimes \ve}\\
M\otimes C\ar [rr]^{M\otimes\Delta}&&M\otimes C\otimes C&&M\otimes I }$$
commute, where the isomorphism is the natural isomorphism of the monoidal structure on $\cat M$.

Let $(M,\rho)$ and $(M',\rho')$ be right $C$-comodules.  A \emph{morphism of right $C$-comodules} from $(M,\rho)$ to $(M',\rho')$ is a morphism $g\in \cat M(M,M')$ such that the diagram
\begin{equation}\label{eqn:modmap}
\xymatrix{M\ar [d]_{\rho}\ar [rr]^g&&M'\ar [d]_{\rho '}\\
M\otimes C\ar [rr]^{g\otimes C}&&M'\otimes C
}
\end{equation}
commutes.  The category of right $C$-comodules and their morphisms is denoted $\cat {Comod}_{C}$.
\end{defn}

\begin{rmk} The forgetful functor $U_{C}:\cat {Comod}_{C}\to \cat M$ admits a right adjoint $-\otimes C: \cat M\to \cat {Comod}_{C}$, where the action map on $X\otimes C$ is given by  
$$X\otimes \Delta:X\otimes C\to X\otimes C\otimes C.$$  
We call $X\otimes C$ the \emph{cofree right $C$-comodule} generated by $X$.
\end{rmk}

\begin{rmk}  It is an easy exercise to show that a morphism $\rho: M\to M\otimes C$  in $\cat M$ is a right $C$-coaction if and only if $\rho$ is a morphism of right $C$-comodules, with respect to the cofree coaction on $M\otimes C$.
\end{rmk}

\begin{rmk}\label{rmk:limits}  If $-\otimes C$ commutes with colimits, e.g., if $\cat M$ is a closed monoidal category, then all colimits exist in $\cat {Comod}_{C}$.  On the other hand, limits do not exist in general in $\cat {Comod}_{C}$.  Since model categories have all finite limits, in order to study the homotopy theory of comodules, we must restrict ourselves to cases in which at least finite limits exist in $\cat {Comod}_{C}$.
\end{rmk}

The category ${}_{C}\cat {Comod}$ of left comodules over a comonoid $C$ and their morphisms is defined analogously, in terms of coaction maps $\lambda:M\to C\otimes M$.    For any object $X$ of $\cat M$, the \emph {cofree left $C$-module} generated by $X$ is $C\otimes X$, endowed with the action map $\Delta\otimes X:C\otimes X\to C\otimes C\otimes X$.

\begin{defn}\label{defn:cotensor} Suppose that $\cat M$ admits equalizers.  Let $(M,\rho)$ and $(N,\lambda)$ be a right and a left $C$-comodule, respectively.  The \emph{cotensor product} $M\underset C\square N$ of $M$ and $N$ is the equalizer
$$M\underset C\square N\to M\otimes N\underset {\rho\otimes Id_{N}}{\overset {M\otimes \lambda }{\rightrightarrows}} M\otimes C\otimes N,$$
which is computed in $\cat M$.
Since this construction is clearly natural in $M$ and in $N$, there is in fact a bifunctor
$$-\underset C\square -:\cat {Comod}_{C}\times {}_{C}\cat {Comod}\to \cat M.$$
\end{defn}

\begin{rmk}\label{rmk:coinvariant}  Let $C$ be a coaugmented comonoid, with coaugmentation $\eta: I\to C$. If $N=I$, endowed with the left $C$-coaction $I\cong I\otimes I \xrightarrow {\eta\otimes I} C\otimes I$, then 
$$M\underset C\square I=\operatorname{equal}(M\underset {\rho}{\overset {M\otimes \eta }{\rightrightarrows}}M\otimes C).$$
In other words $M\underset C\square I$ can be seen as  the object of  \emph{coinvariants} of the coaction $\rho$, justifying the notation 
$$M^{co\, C}:=M\underset C\square I$$
 that we use henceforth.
A similar observation applies to $N^{co\, C}:=I\underset C\square N$ for all $(N,\lambda)\in {}_{C}\cat {Comod}$.
\end{rmk}

\begin{ex}  An easy computation shows that if $C$ is coaugmented and $X\otimes C$ is a cofree $C$-comodule, then 
$$(X\otimes C)^{co\, C}\cong X.$$
\end{ex}

Combining multiplicative and comodule structure, we obtain the theory of comodule algebras.

\begin{defn} \label{defn:comod-monoidal}Suppose that $\cat M$ is symmetric monoidal, and let $(H, \Delta, \mu, \ve, \eta)$ be a bimonoid in $\cat M$.  .  There  is a natural monoidal structure on $\cat{Comod}_{H}$ given by  $(M,\rho)\otimes (M',\rho')=(M\otimes M', \rho *\rho')$, where $\rho*\rho'$ is equal to the composite
$$M\otimes M' \xrightarrow {\rho\otimes \rho'} M\otimes H\otimes M'\otimes H \xrightarrow \cong M\otimes M'\otimes H\otimes H\xrightarrow {M\otimes M'\otimes \mu} M\otimes M'\otimes H.$$ 
The unit object is $I$, endowed with the coaction $I\cong I\otimes I \xrightarrow {I\otimes \eta }I\otimes H$.

Let $\cat{Alg}_{H}$ be the category of  monoids in $\cat{Comod}_{H}$, also known as \emph{$H$-comodule algebras}.  Note that $\cat{Alg}_{H}$ isomorphic to the category of $H$-comodules in the category $\cat{Alg}$ of monoids in $\cat M$.
\end{defn} 

\begin{rmk} Observe that $\cat{Comod}_{I}=\cat M$, while $\cat{Alg} _{I}=\cat{Alg}$.
\end{rmk}

\subsection{Model categories of comodules}

Let $\cat M$ be a model category and a monoidal category. In this section we provide conditions under which the category of comodules over a fixed comonoid in $\cat M$ admits a model category structure inherited from $\cat M$.

We recall the definition of a model category, its homotopy category and derived functors and prove a useful existence theorem for model category structure in the appendix.  We encourage the reader  with questions about the terminology and notation used throughout this paper to consult the appendix.  In particular, we make frequent use of the notions of \emph{right-} and \emph{left-induced model structures} (Definition \ref{defn:induction}). 

Given a model category $\cat M$ that is \emph{cofibrantly generated}  \cite{hovey}, there is a standard procedure for transfering model category structure from $\cat M$ to another category $\cat D$, across an adjunction
$$F:\cat M \adjunct {}{} \cat D:G,$$
where $F$ is the left member of the adjoint pair, under certain conditions on $F$ and $G$ and their relationship to the cofibrations and weak equivalences in $\cat M$ (cf., e.g.,  Lemma 2.3 in \cite{schwede-shipley}).  We cannot apply this technique, however, to transfering model category structure from $\cat M$ to  the category of comodules over a fixed comonoid $C$ in $\cat M$, since the adjoint pair at our disposal is
$$U_{C}:\cat{Comod}_{C}\adjunct {}{}\cat M: -\otimes C,$$ 
where $U_{C}$ is the forgetful functor.  The model category $\cat M$ is on the wrong side of the adjunction for the usual transfer arguments to apply.

In certain special cases it is nonetheless possible to define a model category structure on $\cat{Comod}_{C}$ that is ``inherited'' from that of $\cat M$. We now explore two such special cases. 

\subsubsection{Cartesian categories}\label{subsec:cartesian}  Let $\cat M$ be a category admitting all finite products and a terminal object $e$.  The triple $(\cat M, \times, e)$ is then a monoidal category, of the special type called a \emph{cartesian category}.    

Any object $C$ in a cartesian category $\cat M$ is naturally a comonoid, where the comultiplication is just the usual diagonal morphism $\Delta_{C}:C\to C\times C$.  Moreover, given objects $B$ and $C$ in $\cat M$, the right (or left) $C$-coactions on $B$, with respect to diagonal comultiplication on $C$, are in natural, bijective correspondence with the morphisms in $\cat M$ from $B$ to $C$.

Indeed, if $f\in \cat M(B,C)$, then the composites
$$B\xrightarrow {\Delta_{B}} B\times B \xrightarrow{ B\times f} B\times C$$ 
and 
$$B\xrightarrow {\Delta_{B}} B\times B \xrightarrow{ f\times B} C\times B$$ 
are right and left $C$-coactions on $B$.  Inversely,  if  $\rho: B\to B\times C$ is a right $C$-coaction, then the composite
$$B\xrightarrow\rho B\times C \xrightarrow {pr_{2}} C$$
is an element of $\cat M(B,C)$.  A similar construction works in the case of left $C$-coactions.

Using the universal property of the product, one can easily show that for any right $C$-coaction $\rho: B\to B\times C$,
$$\rho=(B\times pr_{2}\rho)\Delta.$$
It is also immediately obvious that 
$$pr_{2} (B\times f)\Delta =f$$
for all $f\in \cat M(B,C)$.  

Henceforth, let $C$ denote an object of the cartesian category $\cat M$, endowed with its natural diagonal comonoid structure.    The argument above shows that $\cat{Comod}_{C}$ is equivalent to $\cat M/C$, the slice category of objects in $\cat M$ over $C$.   Recall that the objects of $\cat M/C$ are the morphisms in $\cat M$ with target $C$, while a morphism from $f:A\to C$ to $g:B\to C$ is a morphism $a:A\to B$ in $\cat M$ such that $ga=f$ 

It is well known (cf., e.g., Theorem 7.6.5 in \cite {hirschhorn}) that a model category structure on  $\cat M$ gives rise to a model category structure on  $\cat M/C$, in which a morphism 
$$a: (f:A\to C)\to (g:B\to C)$$  
is a weak equivalence, fibration or cofibration  if $a:A\to B$ is a morphism of the same type in $\cat M$.  Thus, in this case, the category of comodules over $C$ does inherit a model structure from $\cat M$, that is right-induced by the forgetful functor.

Important examples of cartesian model categories include the categories of topological spaces,  of simplicial sets and of small categories.

\subsubsection {Postnikov presentations}\label{subsec:postnikov}  We now apply Corollary \ref{cor:postnikov} from the appendix to obtaining model category structure on $\cat{Comod}_{C}$ in the noncartesian case.   All the notation and terminology used below is explained in the appendix.

 The model structure described here is inspired both by the semi-free models of differential modules over differential graded algebras \cite{fht} and by the desire for fibrant replacements of comodules to be ``injective resolutions''.

\begin{thm}\label{thm:post-comod} Let $\cat M$ be endowed with both a model category structure with Postnikov presentation $(\mathsf X, \mathsf Z)$ and a monoidal structure $(\otimes, I)$.  Let $C$ be a comonoid in $\cat M$ such that $\cat {Comod}_{C}$ is finitely bicomplete, and let $U_{C}:\cat{Comod}_{C}\to \cat M$ denote the forgetful functor. Let 
$$\mathsf W=U_{C}^{-1}(\mathsf{WE}_{\cat M}) \text { and } \mathsf C=U_{C}^{-1}(\mathsf{Cof}_{\cat M}).$$   
If $\mathsf{Post}_{\mathsf Z\otimes C}\subset \mathsf W$ and for all $f\in \mor \cat {Comod}_{C}$
there exist 
\begin{enumerate}
\item [(a)] $i\in \mathsf{C} $ and $p\in \mathsf{Post}_{ \mathsf Z\otimes C}$ such
that $f=pi$; 
\smallskip
\item [(b)]$j\in \mathsf{C}\cap \mathsf{W}$ and $q\in \mathsf{Post}_{\mathsf X\otimes C}$ such that
$f=qj$,
\end{enumerate}
then    $\mathsf W$, $\mathsf C$  and $\widehat{\mathsf{Post}_{\mathsf X\otimes C}}$ are the weak equivalences, cofibrations and fibrations in a model category structure on $\cat {Comod}_{C}$, with respect to which 
$$U_{C}:\cat {Comod}_{C} \adjunct {}{}Ê\cat M: -\otimes C$$ 
is a Quillen pair.
\end{thm}

This theorem follows immediately from applying Corollary \ref{cor:postnikov} to the adjunction  $U_{C}:\cat {Comod}_{C}\adjunct{}{} \cat M:-\otimes C$.  We call the factorizations required in hypotheses (a) and (b) \emph{Postnikov factorizations}.

\begin{rmk} In the model category structure developed in Theorem \ref{thm:post-comod}, every $C$-comodule $M$ admits a fibrant replacement $\xymatrix@1{M\wecof &M'\fib &e\otimes C}$, built inductively as follows. There is a ordinal $\lambda$ such that  the limit of a $\lambda$-tower $...\to M'_{\beta +1}\xrightarrow {p_{\beta+1}} M'_{\beta}\to ...$ exists and is isomorphic to $M'$, where for all $\beta<\lambda$, there exist $x_{\beta+1}:X_{\beta+1}\to X_{\beta}$ in $\mathsf X$ and $f_{\beta}: U_{C}M'_{\beta}\to X_{\beta}$ in $\cat M$ such that 
$$\xymatrix{M'_{\beta +1}\ar [d]^{p_{\beta+1}}\ar [r]&X_{\beta+1}\otimes C\ar [d]^{x_{\beta +1}\otimes C}\\ M'_{\beta}\ar [r]^{f_{\beta}^\sharp}&X_{\beta}\otimes C}$$
is a pullback diagram in $\cat{Comod}_{C}$, where $f_{\beta}^\sharp$ is the transpose of $f_{\beta}$.    This is what we think of as an ``injective'' or ``semi-cofree'' resolution of $M$.
\end{rmk}

There are reasonable conditions under which one of the required types of Postnikov factorization exists.  We see in section \ref{subsec:ex-htpycoinv} examples of categories  and comonoids for which these conditions are satisfied.

\begin{lem}\label{lem:postfact}  Let $\cat M$ be endowed with both a model category structure and a monoidal structure $(\otimes, I)$.  Let $C$ be a comonoid in $\cat M$ such that $\cat{Comod}_{C}$ admits pullbacks, and let $U_{C}:\cat{Comod}_{C}\to \cat M$ denote the forgetful functor. Let $ \mathsf C=U_{C}^{-1}(\mathsf{Cof}_{\cat M})$, and let $\mathsf Z$ be a subset of $\mathsf{Fib}_{\cat M}\cap\mathsf{WE}_{\cat M}$ such that  for all $f\in \mor \cat M$,
there exist $j\in \mathsf{Cof}_{\cat M}$ and $q\in \mathsf{Post}_{ \mathsf Z}$ with $f=qj$.

If 
\begin{enumerate}
\item the $C$-coaction morphism $\rho:M\to M\otimes C$ is a cofibration in $\cat M$ for every $(M,\rho)\in \ob \cat {Comod}_{C}$,
\item $-\otimes C:\cat M\to \cat M$ preserves weak equivalences and cofibrations, 
\item for all $i:M\to N$ in $\mathsf C$ and all morphisms $g:M\to N'$ in $\cat{Comod}_{C}$, the induced map $(i,g):M\to N\times N'$ is in $\mathsf C$, and
\item $\mathsf{Post}_{ \mathsf Z}\otimes C\subset \mathsf{Post}_{ \mathsf Z\otimes C}$,
\end{enumerate}
then  for all $f\in \mor \cat {Comod}_{C}$,
there exist $i\in \mathsf{C} $ and $p\in \mathsf{Post}_{ \mathsf Z\otimes C}$ such
that $f=pi$.
\end{lem}

Note that hypothesis (4) holds if, for example, $\mathsf Z=\mathsf{Fib}_{\cat M}\cap\mathsf{WE}_{\cat M}$, since then $\mathsf{Post}_{ \mathsf Z}=\mathsf Z$.

\begin{proof}  Let $f:M\to N$ be any morphism of right $C$-comodules.   Let $e$ denote the terminal object in $\cat M$, and consider the factorization in $\cat M$
$$\xymatrix{U_{C}M\ar [rr]\ar [dr]_{j}&&e\\ &Z\ar [ur]_{q}^\sim},$$
where $j$ is a cofibration and $q\in\mathsf{Post}_{ \mathsf Z}$.  Taking adjoints, we obtain a commuting triangle in $\cat{Comod}_{C}$
$$\xymatrix{M\ar [rr]\ar [dr]_{j^\sharp}&&e\otimes C\\ &Z\otimes C\ar [ur]_{q\otimes C}^\sim},$$
where $e\otimes C$ is the terminal object in $\cat {Comod}_{C}$, since $-\otimes C$ preserves limits.  Moreover, $j^\sharp=(j\otimes C)\rho$, where $\rho:M\to M\otimes C$ is the $C$-coaction on $M$.  It follows from hypotheses (1) and (2) that $U_{C}j^\sharp$ is a cofibration in $\cat M$, i.e., that $j^\sharp\in \mathsf C$.

Hypothesis (3) now implies that $i=(j^\sharp, f): M\to (Z\otimes C)\times N$ is also in $\cat C$.  
Finally, consider the pullback
$$\xymatrix{(Z\otimes C)\times N \ar [r]\ar [d]^{ p}&Z\otimes C\ar [d] ^{q\otimes C}\\ N\ar [r]& e\otimes C.}$$
Since $q\otimes C\in \mathsf{Post}_{ \mathsf Z}\otimes C\subset \mathsf{Post}_{ \mathsf Z\otimes C}$, the induced map $ p:(Z\otimes C)\times N \to N$ is an element of $\mathsf{Post}_{\mathsf Z\otimes C}$ as well and  $f=pi$ is the desired factorization.
\end{proof} 

It is generally more difficult to prove the existence of the second sort of Postnikov factorization in $\cat {Comod}_{C}$.  Rather than establishing a general result, we show that such factorizations exist in the examples we treat in section \ref{subsec:ex-htpycoinv}.  We suspect that the methods of proof we apply in these specific cases can be generalized in a relatively straightforward manner, whenever it is possible to construct Postnikov-type decompositions of objects in $\cat M$ inductively.

\subsection{Model categories of comodule algebras}\label{subsec:comod-alg}

Let $(\cat M, \otimes, I)$  be a monoidal category that is endowed with model category structure as well, and let $(H, \Delta, \mu, \ve, \eta)$ be a bimonoid in $\cat M$.  We now analyze possible model category structures on $\cat{Alg}_{H}$, the category of $H$-comodule algebras in $\cat M$.  
As above, we separate the analysis into two parts:  the cartesian case and the Postnikov case.

\subsubsection{Cartesian categories}  Let $(\cat M, \times , e)$ be a cartesian category and a model category.  If $A$ is a monoid in $\cat M$, then the diagonal map $A\to A\times A$ is a morphism of monoids, as can be seen by a straightforward application of the universal property of the product. 

The argument in section \ref{subsec:cartesian} implies that if $H$ is a bimonoid in $\cat M$, with comultiplication equal to the diagonal map, then the category of $H$-comodule algebras in $\cat M$ is isomorphic to the slice category $\cat{Alg}/H$ of monoid maps with target $H$.  A model structure on $\cat {Alg}$ therefore naturally gives rise to a right-induced model structure on the category of $H$-comodule algebras,  given by $\mathsf C_{\cat{Alg}_{H}}=(U'_{H})^{-1}(\mathsf C_{\cat {Alg}})$ for each of the distinguished classes $\mathsf C=\mathsf {WE}, \mathsf{Fib}, \mathsf{Cof}$, where $U'_{H}: \cat{Alg}_{H}\to \cat{Alg}$ is the forgetful functor \cite{hirschhorn}.  It remains for us to specify the model structures on $\cat {Alg}$ that interest us.

\begin{rmk}\label{rmk:cartesian-mon} Theorem 4.1 in \cite {schwede-shipley} implies that if $(\cat M, \times, e)$ is a cofibrantly generated, monoidal model category satisfying the monoid axiom and if every object in $\cat M $ is small with relative to $\cat M$, then $ \cat{Alg}$ admits a cofibrantly generated model structure that is right-induced by the forgetful functor $U_{Alg}: \cat{Alg} \to \cat M$. For example, as mentioned in section 5 of \cite{schwede-shipley}, the cartesian category $(\cat{sSet},\times , *)$ of simplicial sets satisfies these criteria.
\end{rmk}

\subsubsection{Postnikov presentations} \label{subsec:postnikov-comodalg} Let $(\cat M, \otimes, I)$ be a cofibrantly generated, monoidal model category that satisfies the monoid axiom and such that all objects are small relative to $\cat M$. Let $\mathcal I$ denote the generating cofibrations of $\cat M$.  Let $F:\cat M\to \cat {Alg}$ denote the free monoid functor, i.e., $F(X)=\coprod_{n\geq 0}X^{\otimes n}$, endowed with the multiplication induced by the isomorphism $X^{\otimes m}\otimes X^{\otimes n}\cong X^{\otimes m+n}$. There is an adjoint pair
$$F:\cat M \adjunct {}{} \cat {Alg}: U_{Alg},$$
where $U_{Alg}$ is the forgetful functor.  Theorem 4.1 in \cite{schwede-shipley} implies that there is a cofibrantly generated, right-induced model category structure on $\cat {Alg}$ where $\mathsf{Cof}_{\cat {Alg}}$ is generated by $F(\mathcal I)$.

Let $H$ be a bimonoid in $\cat M$. There is a free/forgetful adjoint pair
$$F_{H}:\cat {Comod}_{H} \adjunct {}{} \cat {Alg}_{H}: U_{Alg,H},$$
similar to the pair $(F,U_{Alg})$ above, where $F_{H}$ is defined in terms of the monoidal structure on $\cat{Comod}_{H}$ given in Definition \ref{defn:comod-monoidal}.
Unfortunately, the model category structure on $\cat {Comod}_{H}$ obtained in Theorem \ref{thm:post-comod} is not generally cofibrantly generated, so that we cannot directly apply the results of \cite{schwede-shipley} to defining a model structure on $\cat {Alg}_{H}$.  It would interesting to determine conditions under which $U_{Alg, H}$ does right-induce a model structure on $\cat {Alg}_{H}$.  For example one could specify conditions under which the model category structure on $\cat {Comod}_{H}$ is cofibrantly generated, perhaps by $U_{H}(\op I)$ and $U_{H}(\op J)$, where $\op I$ and $\op J$ are the generating cofibrations and generating acyclic cofibrations in $\cat M$. 

The forgetful/cofree adjoint  pair
$$U'_{H}:\cat {Alg}_{H}\adjunct {}{} \cat {Alg}:-\otimes H$$
can also give rise to an interesting model category structure on $\cat {Alg}_{H}$.  We cannot apply standard tranfer techniques, since the cofibrantly generated model category is on the right side of this adjunction, so we again appeal to Corollary \ref{cor:postnikov}, obtaining the following result.

\begin{thm}\label{thm:post-comodalg} Let $(\cat M, \otimes, I)$ be a cofibrantly generated, monoidal model category that satisfies the monoid axiom and such that all objects are small relative to $\cat M$.  Let $H$ be a bimonoid in $\cat M$ such that $\cat{Alg}_{H}$ is finitely bicomplete, and let $$U'_{H}:\cat {Alg}_{H}\adjunct {}{} \cat {Alg}:-\otimes H$$ denote the forgetful/cofree adjoint functor pair. Let 
$$\mathsf W=(U'_{H})^{-1}(\mathsf{WE}_{\cat {Alg}}) \text { and } \mathsf C=(U'_{H})^{-1}(\mathsf{Cof}_{\cat{Alg}}),$$
and
$$\mathsf X=\mathsf{Fib}_{\cat {Alg}} \text { and } \mathsf Z=\mathsf{Fib}_{\cat {Alg}}\cap\mathsf{WE}_{\cat {Alg}},$$
where $\cat {Alg}$ is endowed with the model structure right induced by the forgetful functor $U_{Alg}:\cat {Alg}\to \cat M$.
 
If $\mathsf{Post}_{\mathsf Z\otimes H}\subset \mathsf W$ and for all $f\in \mor \cat {Alg}_{H}$
there exist 
\begin{enumerate}
\item [(a)] $i\in \mathsf{C} $ and $p\in \mathsf{Post}_{ \mathsf Z\otimes H}$ such
that $f=pi$; 
\smallskip
\item [(b)]$j\in \mathsf{C}\cap \mathsf{W}$ and $q\in \mathsf{Post}_{\mathsf X\otimes H}$ such that
$f=qj$,
\end{enumerate}
then    $\mathsf W$, $\mathsf C$  and $\widehat{\mathsf{Post}_{\mathsf X\otimes H}}$ are the weak equivalences, cofibrations and fibrations in a model category structure on $\cat {Alg}_{H}$, with respect to which 
$$U'_{H}:\cat {Alg}_{H} \adjunct {}{}Ê\cat {Alg}: -\otimes H$$ 
is a Quillen pair.
\end{thm}

In section \ref{subsec:ex-htpycoinv} we examine examples of categories and bimonoids that satisfy the hypotheses of this theorem.

\section{Homotopy coinvariants}

Let $C$ be a comonoid in a monoidal model category $\cat M$. In this section we define and provide several examples of a homotopy invariant replacement of the coinvariants functor
$$\operatorname{Coinv}: \cat{Comod}_{C}\to \cat M: M\to M\underset C\square I.$$
Our strategy is to determine conditions under which the coinvariants functor is the right member of a Quillen pair, then to define the homotopy coinvariants functor to be the total derived right functor of $\operatorname{Coinv}$ under those conditions.

\subsection {Deriving the coinvariants functor}

\begin{defn} Let $(\cat M,\otimes , I)$ be a monoidal category, and let $C$ be a comonoid in $\cat M$ endowed with a coaugmentation $\eta:I\to C$.  The \emph{trivial comodule functor}  $\operatorname{Triv}: \cat M\to \cat{Comod}_{C}$ is specified by $\operatorname{Triv}(X)=(X, X\otimes \eta)$ for all objects $X$ in $\cat M$ and $\operatorname{Triv}(f)=f$ for all morphisms $f$.
\end{defn}

Note that $\cat M$ could itself be the category of a monoids in an underlying monoidal category, i.e., the case of  comodule algebras is englobed by this definition.

\begin{rmk} It is easy to check that $\operatorname{Triv}:\cat M \to \cat{Comod}_{C}$ is left adjoint to the coinvariants functor
$$\operatorname{Coinv}:\cat {Comod}_{C}\to \cat M: (M,\rho)\mapsto M^{co \, C}=M\underset C\square I.$$
\end{rmk}

\begin{defn} Let $C$ be a coaugmented comonoid in a monoidal category $\cat M$.  If $\operatorname{Triv}:\cat M \adjunct {}{} \cat{Comod}_{C}: \operatorname{Coinv}$ is a Quillen pair, then the total right derived functor
$$\mathbb R\operatorname{Coinv}: \operatorname{Ho}\cat {Comod}_{C}\to \operatorname{Ho}\cat M$$
is the \emph{homotopy coinvariants functor}.   If $M$ is a right $C$-comodule, then a representative of $\mathbb R\operatorname{Coinv} (M)$ is called a \emph{model} of the homotopy coinvariants of $M$.
\end{defn}

\begin{notn} In a slight abuse of notation, any model of the homotopy coinvariants of a right $C$-comodule $M$ is denoted $M^{hco\, C}$.  Thus, if $\xymatrix@1{M\wecof &RM\fib &e\otimes C}$ is any fibrant replacement of $M$ in $\cat{Comod}_{C}$, where $e$ is the terminal object in $\cat M$, then $(RM)^{co\, C}=M^{hco\, C}$. 
\end{notn}

In the following propositions, we specify conditions under which $(\operatorname{Triv}, \operatorname{Coinv})$ is a Quillen pair and which therefore guarantee the existence of a homotopy coinvariants functor.  We first consider the cartesian case.

\begin{prop} Let $(\cat M, \times, e)$ be a cartesian category and a model category.  If $C$ is any object in  $\cat M$, seen as a comonoid via the diagonal map $\Delta:C\to C\times C$ and endowed with a coaugmentation $\eta: e\to C$, then the adjoint pair 
$$\operatorname{Triv}:\cat M \adjunct {}{} \cat{Comod}_{C}: \operatorname{Coinv}$$
is a Quillen pair, where $\cat {Comod}_{C}$ is endowed with the model structure described in section \ref{subsec:cartesian}.
\end{prop} 

\begin{proof} Since $\cat {Comod}_{C}$ is isomorphic to $\cat M/C$, this proposition follows immediately from the definition of the model category structure on $\cat M/C$ (cf. section \ref{subsec:cartesian}).
\end{proof}

As a special case of the proposition above, we can treat coinvariants of  comodule algebras.

\begin{cor} \label{cor:comodalg-cartesian}Let $(\cat M, \times, e)$ be a cartesian category and a model category such that the forgetful functor $\cat {Alg}\to \cat M$ right-induces a model structure on $\cat {Alg}$.  If $H$ is a bimonoid in $\cat M$, with comultiplication equal to the diagonal map, then the adjoint pair 
$$\operatorname{Triv}:\cat {Alg} \adjunct {}{} \cat{Alg}_{H}: \operatorname{Coinv}$$
is a Quillen pair.
\end{cor}

We now consider the case of model categories of comodules with left-induced model structures, as in Theorem \ref{thm:post-comod}.

\begin{prop}\label{prop:htpycoinv-dual} Let $\cat M$ be endowed with both a model category structure  and a monoidal structure $(\otimes, I)$.  If $\cat {Comod}_{C}$ admits a model category structure left-induced by $U_{C}:\cat {Comod}_{C} \to\cat M$,  then  
$$\operatorname{Triv}:\cat M \adjunct {}{} \cat{Comod}_{C}: \operatorname{Coinv}$$
is a Quillen pair as well.
\end{prop}

\begin{proof}  Since $U_{C}$ left-induces the model structure on $\cat{Comod}_{C}$, it is clear that $\operatorname{Triv}$ preserves both cofibrations and weak equivalences.
\end{proof}

We obtain a result for left-induced model structures on $\cat {Alg}_{H}$ (cf. Theorem \ref{thm:post-comodalg})  as a special case of the proposition above.

\begin{cor} \label{prop:comodalg-dual} Let $(\cat M, \otimes, I)$ be a monoidal category endowed  with a  model category structure.  Let  $H$ be a bimonoid in  $\cat M$.   If $\cat {Alg}$ admits a model structure right-induced by the forgetful functor $U_{Alg}:\cat {Alg}\to \cat M$ and  $\cat {Alg}_{H}$ admits a model structure left-induced by the forgetful functor $U_{H}':\cat {Alg}_{H}\to  \cat{Alg}$,  then  
$$\operatorname{Triv}:\cat {Alg} \adjunct {}{} \cat{Alg}_{H}: \operatorname{Coinv}$$
is a Quillen pair as well.
\end{cor}

When the model structure on $\cat {Alg}_{H}$ is right-induced, we have the following result.

\begin{prop} \label{prop:comodalg-dual}Let $(\cat M, \otimes, I)$ be a cofibrantly generated, monoidal model category that satisfies the monoid axiom and such that all objects are small relative to $\cat M$.  Let $H$ be a bimonoid in $\cat M$.  

If $\cat {Comod}_{H}$ admits a model category structure left-induced by $U_{H}:\cat {Comod}_{H} \to\cat M$ and $\cat {Alg}_{H}$ admits a model structure right-induced by  $U_{Alg, H}:\cat {Alg}_{H}\to \cat{Comod}_{H}$,  then  
$$\operatorname{Triv}:\cat {Alg} \adjunct {}{} \cat{Alg}_{H}: \operatorname{Coinv}$$
is a Quillen pair as well, where $\cat {Alg}$ is endowed with its usual right-induced model structure.
\end{prop}

\begin{proof}  It is easy to check that the following diagram of functors commutes.
$$\xymatrix{\cat M\ar [d]_{\operatorname{Triv}}\ar[r]^F& \cat{Alg}\ar [d]^{\operatorname{Triv}}\\ \cat{Comod}_{H}\ar [r]^{F_{H}}& \cat{Alg}_{H}}$$
Here, $F$ and $F_{H}$ are the free monoid functors.  

Let $\mathcal I$ be the set of generating cofibrations in $\cat M$, and let $i\in \mathcal I$.  Proposition \ref{prop:htpycoinv-dual} implies that $\operatorname{Triv}(i)$ is a cofibration in $\cat{Comod}_{H}$.  Moreover, a simple adjunction argument shows that $F_{H}$ preserves cofibrations, since its right adjoint, the forgetful functor $U_{Alg, H}: \cat{Alg}_{H}\to \cat{Comod}_{H}$, right-induces the model structure on $\cat{Alg}_{H}$.  Thus, $F_{H}\circ\operatorname{Triv}(i)$ is a cofibration in $\cat{Alg}_{H}$ for all $i\in \mathcal I$, i.e.,
$$\operatorname{Triv}(F\mathcal I)\subset \mathsf{Cof}_{\cat {Alg}_{H}}.$$
Since $F\mathcal I$ generates the cofibrations in $ \cat{Alg}$ and $\operatorname{Triv}$ is a left adjoint, it follows that
$$\operatorname{Triv}(\mathsf{Cof}_{\cat {Alg}})\subset \mathsf{Cof}_{\cat {Alg}_{H}}.$$

A similar argument applied to the set of generating acyclic cofibrations in $\cat M$ implies that $\operatorname{Triv}: \cat{Alg}\to \cat{Alg}_{H}$ preserves acyclic cofibrations as well.  We conclude that $(\operatorname{Triv}, \operatorname{Coinv})$ is indeed a Quillen pair.
\end{proof}

\subsection{Examples}\label{subsec:ex-htpycoinv}

We present in this section four examples of categories of comodules in which there is a good definition of homotopy coinvariants.  

\subsubsection{Spaces}

Let $\cat M=\cat {Top}$ or $\cat {sSet}$, with their cartesian monoidal structure.  We refer to the objects of either category as \emph {spaces}.  

Let $Y$ be a  space, seen as  a comonoid via the diagonal map.   A coaugmentation $\eta: * \to Y$ consists of a choice of basepoint $y_{0}=\eta (*)$ for the space $Y$.  Let $X$ be another space, and let $f\in \cat{M}(X,Y)$, giving rise to a right $Y$-coaction on $X$ as in section \ref{subsec:cartesian}.

To compute the homotopy coinvariants of $X$ with respect to the coaction induced by $f$, we first find a fibrant replacement of $f:X\to Y$ in the category $\cat{M}/Y$.  Since the identity map on $Y$ is the terminal object in $\cat M/ Y$,  a fibrant replacement of $f$ consists of a commutative diagram in $\cat{M}$
$$\xymatrix{X\wecof \ar [dr]_{f}& X'\fib ^{p}\ar [d]^p&Y. \ar [dl]^{=}\\&Y}$$
A model of the homotopy coinvariants of $X$ is then
$$X^{hco\, Y}=(X')^{co\, Y}=\operatorname{equal}(X'\egal {(X'\times p)\Delta}{(X'\times \eta)} X'\times Y)= p^{-1}(y_{0}).$$
In other words, the space of homotopy coinvariants of the coaction induced by $f$ is exactly the homotopy fiber of $f$.

\subsubsection{Simplicial monoids}\label{subsec:simplmon} As mentioned in Remark \ref{rmk:cartesian-mon}, the category $\cat{sMon}$ of simplicial monoids admits a cofibrantly generated model structure that is right-induced by the forgetful functor $U_{Alg}:\cat {sMon}\to \cat {s Set}$.  Let $H$ be a simplicial monoid, seen as a bimonoid via the diagonal map. Let $A$ be another simplicial monoid, and let $f\in \cat{sMon}(A,H)$, giving rise to a right $H$-coaction on $A$.

Since the identity map on $H$ is the terminal object in $\cat {sMon}/ H$,  a fibrant replacement of $f$ consists of a commutative diagram in $\cat{sMon}$
$$\xymatrix{A\wecof \ar [dr]_{f}& A'\fib ^{p}\ar [d]^p&H. \ar [dl]^{=}\\&H},$$
i.e., $p$ is a simplicial homomorphism and the underlying map of simplicial sets is a Kan fibration, since the model structure of $\cat{sMon}$ is right-induced.  A model of the homotopy coinvariants of $A$ is then
$$A^{hco\, H}=(A')^{co\, H}=\operatorname{equal}(A'\egal {(A'\times p)\Delta}{(A'\times \eta)} A'\times H)= p^{-1}\big(\eta(e)\big).$$
Note that the equalizer is computed in $\cat{sMon}$, but that the forgetful functor to simplicial sets preserves products and equalizers, since it is a right adjoint. In other words, the simplicial monoid of homotopy coinvariants of the coaction induced by $f$ is exactly the homotopy fiber of $f$, which is a simplicial submonoid of $A'$.

\subsubsection{Chain complexes}\label{subsec:chcx}

For the sake of simplicity, we work here over a field $\Bbbk$, though our results probably hold over a principal ideal domain $R$, as long as the comonoid we consider is $R$-flat.  

Let $\cat M=\cat{Ch}_{\Bbbk}$, the category of nonnegatively graded chain complexes of \emph{finite dimensional} $\Bbbk$-vector spaces, also known as \emph{finite-type} chain complexes. There is a model structure on this category for which $\mathsf{WE}_{\cat M}$ is the set of quasi-isomorphisms (chain maps inducing isomorphisms in homology), $\mathsf{Fib}_{\cat M}$ is the set of chain maps that are surjective in positive degrees and cofibrations are degreewise injections \cite{hovey}.   Endowed with the usual tensor product of chain complexes, $\cat M$ is a monoidal model category, satisfying the monoid axiom.   The unit of the tensor product is just $\Bbbk$ itself, considered as a chain complex concentrated in degree 0.   

For any chain complex $X$, finite products of chain complexes commute with $-\otimes X$, since a finite product of chain complexes is isomorphic to the finite coproduct of the same chain complexes, and $-\otimes X$ commutes with colimits. Furthermore, since we are working over a field and therefore $-\otimes X$ is left and right exact, equalizers commute with $-\otimes X$ as well.  We conclude that all finite limits commute with $-\otimes X$.

Let $C$ be a comonoid in $\cat M$, i.e.,  a chain (or dg) coalgebra.    Since all finite limits commute with $-\otimes C$, the category $\cat{Comod}_{C}$ of $C$-comodules is finitely complete.  Furthermore, $U_{C}$ creates colimits in $\cat{Comod}_{C}$, since $-\otimes C$ commutes with colimits,  so $\cat{Comod}_{C}$ is cocomplete as well.

A slightly more general existence result holds for limits in $\cat {Comod}_{C}$.  Given a family $\mathfrak M=\{M_{\frak a}\mid\mathfrak a\in \mathfrak A\}$ of $C$-comodules such that 
$$\sum _{\mathfrak a\in \mathfrak A} \dim (M_{\mathfrak a})_{n}<\infty$$
for all $n$,
the product of chain complexes  $\prod_{\mathfrak a\in \mathfrak A}M_{\mathfrak a}$ is a degreewise direct sum. Therefore, $(\prod_{\mathfrak a\in \mathfrak A}M_{\mathfrak a})\otimes C \cong \prod_{\mathfrak a\in \mathfrak A}(M_{\mathfrak a}\otimes C)$.  It follows that $\prod_{\mathfrak a\in \mathfrak A}M_{\mathfrak a}$ admits a natural comultiplication with respect to which it is the product of the family $\mathfrak M$ in $\cat{Comod}_{C}$.

For all $n\geq 1$, let $D^n$ denote the chain complex that is $\Bbbk$-free on a generator $x_{n-1}$ in degree $n-1$ and on a generator $y_{n}$ in degree $n$, with differential $d$ satisfying $dy_{n}=x_{n-1}$.  For all $n\geq 0$, let $S^{n}$ denote the chain complex that is $\Bbbk$-free on one generator $y_{n}$ of degree $n$, and let $p_{n}:D^n\to S^n$ denote the obvious projection map.  It is not difficult to check that $\cat M$ admits a Postnikov presentation (Definition \ref{defn:postnikov}), where 
$$\mathsf X=\{p_{n}:D^n\to S^n\mid n\geq 1\} \cup \{p_{0}:0 \to S^0\}\cup \{p_{n}':S^n\to 0\mid n\geq 0\} $$
and 
$$\mathsf Z=\{q_{n}:D^n\to 0\mid n\geq 1\}.$$

The goal of this section is to prove the following existence result.

\begin{thm}\label{thm:modcat-chcomod} Let $C$ be a chain coalgebra that is $1$-connected (i.e., $C_{0}=\Bbbk$ and $C_{1}=0$) and coaugmented.  The forgetful functor $U_{C}:\cat {Comod}_{C}\to \cat M$ left-induces a model structure on $\cat{Comod}_{C}$ with Postnikov presentation $(\mathsf X\otimes C, \mathsf Z\otimes C)$.
\end{thm}

Our strategy for proving this theorem is to show that the hypotheses of Theorem \ref{thm:post-comod} are satisfied.  In this proof we use freely  the terminology and results of section \ref{subsec:induced}. 

Note that since $\mathsf W=U_{C}^{-1}(\mathsf{WE}_{\cat M})$, a morphism of $C$-comodules is in $\mathsf W$ if and only if it is a quasi-isomorphism.  Similarly, a morphism of $C$-comodules is an element of $\mathsf C=U_{C}^{-1}(\mathsf{Cof}_{\cat M})$ if and only if it is degreewise injective.

In proving that the hypotheses of Theorem \ref{thm:post-comod} are satisfied, we need to compute the homology of certain inverse limits of towers of chain complexes.  Since the homology of an inverse limit is in general not isomorphic to the inverse limit of the homology groups, we call upon the following useful, classical result, which is an immediate consequence of Proposition 3.5.7 and Theorem 3.5.8 in  \cite {weibel}.

Let $\cat C$ denote either the category of chain complexes or the category $\cat{Ab}$ of abelian groups. Recall that a tower 
$$X:\lambda ^{op} \to \cat C,$$
i.e., 
$$...\to X_{\beta +1}\xrightarrow {p_{\beta +1}} X_{\beta}\to \cdots \xrightarrow {p_{1}} X_{0}$$  satisfies the \emph{Mittag-Leffler condition} if for all $\beta<\lambda$, there exists $\gamma >\beta$ such that the image of the composite $X_{\gamma'}\to X_{\beta}$ is equal to the image of $X_{\gamma}\to X_{\beta}$ for all $\gamma'>\gamma$.  For example, if $p_{\beta}$ is a surjection for all $\beta$, then the tower satisfies the Mittag-Leffler condition.

\begin{thm}\label{thm:weibel} Let $X:\lambda ^{op} \to \cat M$ be a tower of chain complexes satisfying the Mittag-Leffler condition.  If the induced tower $H_{n}X:\lambda ^{op }\to \cat {Ab}$ also satisfies the Mittag-Leffler condition, for all $n\geq 0$, then  $\lim _{\lambda^{op}}X_{\beta}\to X_{0}$ is a quasi-isomorphism.
\end{thm}

In particular, if $X$ is a tower of surjective quasi-isomorphisms, then $\lim _{\lambda^{op}}X_{\beta}\to X_{0}$ is a quasi-isomorphism.

To prove that $\mathsf{Post}_{\mathsf Z\otimes C}\subset \mathsf W$, observe first that every element $q_{n}: D^{n}\to 0$ of $\mathsf Z$ is a surjective quasi-isomorphism, and therefore $q_{n}\otimes C$ is a surjective quasi-isomorphism as well. It follows that for all $C$-comodules $M$, the projection map $M\times (D^{n}\otimes C)\to M$, which comes from pulling back $M\to 0$ and $q_{n}\otimes C$, is also a surjective quasi-isomorphism.  Thus, in any $\mathsf Z\otimes C$-Postnikov tower 
$$...\to M_{\beta+1}\xrightarrow{\bar q_{\beta +1}}M_{\beta}\to ...,$$
each $\bar q_{\beta+1}$ is a surjective quasi-isomorphism, and therefore the composition of the tower $\lim _{\lambda^{op}} M_{\beta}\to M_{0}$, if it exists, is also a quasi-isomorphism.

We now establish the existence of the required Postnikov factorizations, i.e., hypotheses (a) and (b) of Theorem \ref{thm:post-comod}.  Hypothesis (a) is proved by showing that the hypotheses of Lemma \ref{lem:postfact} are satisfied.  We begin by noting that since $(\mathsf X, \mathsf Z)$ is a Postnikov presentation of $\cat M$, for all $f\in \mor \cat M$,
there exist $j\in \mathsf{Cof}_{\cat M}$ and $q\in \mathsf{Post}_{ \mathsf Z}$ with $f=qj$.

Hypothesis (1) of Lemma \ref{lem:postfact} is satisfied, since, for every $C$-comodule $(M, \rho)$, the composite $(M\otimes \ve)\rho$ must be the identity on $M$, and therefore $\rho$ must be injective. Hypothesis (2) holds since we are working over a field, while hypothesis (3) follows from the observation that if $i$ is a degreewise injective map of comodules and $g$ is any map of comodules, then $(i,g)$ is necessarily degreewise injective.  Finally, observe that $\mathsf {Post}_{\mathsf Z}$ consists of projection maps $X\times \prod_{\beta <\lambda} D^{n_{\beta}}\to X$ for various ordinals $\lambda$ and various objects $X$ in $\cat M$.  Note that since we are working with finite-type complexes, for any $n$ the set $\{\beta<\lambda\mid n_{\beta }=n\}$ is finite. On the other hand, by our finite-type assumption,
$$\big(X\times \prod_{\beta <\lambda} D^{n_{\beta}}\big)\otimes C\cong (X\otimes C)\times  \prod_{\beta <\lambda} (D^{n_{\beta}}\otimes C),$$
which is obviously an element in $\mathsf {Post}_{\mathsf Z\otimes C}$.  Thus, hypothesis (4) is also satisfied, and therefore Lemma \ref{lem:postfact} holds.

We prove the existence of the second sort of Postnikov factorization (hypothesis (b) of Theorem \ref{thm:post-comod}) by an inductive argument, which is essentially dual to the usual argument for the existence of semi-free models for modules over a chain algebra \cite{fht}, though one has to be careful.  The argument is formulated in terms of a certain successive approximations to the weak equivalences in $\cat M$ and $\cat {Comod}_{C}$, defined as follows.

\begin{defn} Let $n\in \mathbb N$.  An \emph{$n$-equivalence} in $\cat M$ is a morphism of chain complexes $f:X\to Y$ such that $H_{*}f$ is degreewise injective and $H_{k} f$ is an isomorphism for all $k\leq n$. A morphism $g:M\to N$ of $C$-comodules is a \emph{$n$-equivalence} if $U_{C}g$ is.
\end{defn}

It is obvious that a weak equivalence is an $n$-equivalence for all $n$.

The next lemma is the base step of the inductive proof of hypothesis (b) of Theorem \ref{thm:post-comod}.

\begin{lem} Let $C$ be a $1$-connected chain coalgebra. If $f:M\to N$ is any morphism of $C$-comodules, then there exists a degreewise-injective $0$-equivalence $i:M\to M'$ and a map  $p:M'\to N$ in $\mathsf{Post}_{\mathsf X\otimes C}$ such that
$f=pi$.
\end{lem}

\begin{proof} Observe first that $f:M\to N$ factors as 
$$M\xrightarrow{(\rho, f)} (M\otimes C)\times N\xrightarrow {p'} N,$$
where $p'$ is the projection map. Since $U_{C}M\to 0$ is necessarily the composition of an $\mathsf X$-Postnikov tower, and $-\otimes C$ commutes with degreewise-finite limits,  $p'\in \mathsf{Post}_{\mathsf X\otimes C}$.

Let $N'= (M\otimes C)\times N$, and let $i'=(\rho, f)$, which is degreewise injective and induces a degreewise injection in homology, since $H_{*}N'\cong (H_{*}M\otimes H_{*}C)\oplus H_{*}N$.  Let $K$ denote the cokernel of $i'$. Note that since $U_{C}$ is a left adjoint and therefore preserves colimits, the chain complex underlying $K$ is the cokernel of the chain map underlying $f$.  

Viewing $K_{0}$ and $H_{0}K$ as chain complexes concentrated in degree 0, consider the sequence of quotient maps
$$N'\to K\to K_{0}\to H_{0}K,$$
the composite of which is a chain map, denoted $k:N'\to H_{0}K$.  Let 
$$k^\sharp:N'\to H_{0}K\otimes C$$ 
denote the corresponding comodule map. Note that $H_{0}K\otimes C$ is isomorphic to $\prod_{\mathfrak a\in \mathfrak A}S^0$, where $\mathfrak A$ is a basis of $H_{0}K$.

Consider the pullback diagram
$$\xymatrix{M'\ar [d]^{p''}\ar[r]& 0\ar[d]\\ N'\ar [r]^(0.4){k^\sharp}& H_{0}K\otimes C}$$
in $\cat {Comod}_{C}$.  The fact that $0\to H_{0}K\otimes C$ is a product of maps in $\mathsf X\otimes C$ implies that $p''\in \mathsf{Post}_{\mathsf  X\otimes C}$.

To conclude, we let $i:M\to M'$ be the morphism induced by the pair $i':M\to N'$ and $0:M\to 0$, and we let $p=p'p''$.  It is easy to check that $i$ is degreewise injective and a $0$-equivalence.  Moreover, $p\in  \mathsf{Post}_{\mathsf  X\otimes C}$, since it is the composite of two $\mathsf X\otimes C$-Postnikov towers.  Finally, $f=pi$, as desired.
\end{proof}

The inductive step of the argument proceeds as follows.  

\begin{lem}\label{lem:inductive-step} Let $C$ be a $1$-connected chain coalgebra, and let $n\in \mathbb N$. If $f:M\to N$ is  a degreewise-injective $n$-equivalence of $C$-comodules, then there exists a degreewise injective $(n+1)$-equivalence $i:M\to M'$ and a map  $p:M'\to N$ in $\mathsf{Post}_{\mathsf X\otimes C}$ such that
$f=p i$. 
\end{lem}

\begin{proof}  Let $K$ denote the cokernel of $f$, computed in $\cat{Comod}_{C}$.  As in the previous proof, the chain complex underlying $K$ is the cokernel of the chain map underlying $f$.  Considering the long exact sequence in homology induced by the short exact sequence of complexes $0\to M\xrightarrow f N\xrightarrow{} K\to 0$, we see that $H_{k}K=0$ for all $k\leq n$.

Let $Z_{n+1}K$ denote the subspace of cycles in $K$ of degree $n+1$.  Since we are working over a field, we can choose a section $\sigma _{n+1}:K_{n+1}\to Z_{n+1}K$.

Viewing $K_{n+1}$, $Z_{n+1}K$ and $H_{n+1}K$ are chain complexes concentrated in degree $n+1$, consider the sequence of linear maps
$$N\to K\to K_{n+1}\xrightarrow {\sigma_{n+1}} Z_{n+1}K\to H_{n+1}K,$$
where maps other than $\sigma_{n+1}$ are the obvious quotient maps. Let $k: N\to H_{n+1}K$ denote the composite of this sequence, which is a chain map, and let $k^\sharp:N\to H_{n+1}K\otimes C$ be the corresponding morphism of $C$-comodules.

Let $(X,d)$ be a chain complex such that $H_{0}(X,d)=0$. Recall that there is a ``based path object'' construction on $(X,d)$, which is an acyclic chain complex denoted $P(X,d)$, together with a fibration $q:P(X,d)\to (X,d)$.  More precisely, $P(X,d)=(X\oplus s^{-1} (X_{\geq 2}\oplus (\ker d)_{1}), D)$, where $X_{+}$ denotes the positive-degree part of $X$, $s^{-1}$ denotes desuspension, and $q$ is the obvious quotient map .  The differential $D$ is specified by $D(x)=dx-s^{-1}x$ and $D(s^{-1} x)=-s^{-1} dx$.  

Note that the projection $q:P(H_{n+1}K,0)\to H_{n+1}K$ is an element of $\mathsf{Post}_{\mathsf X}$, since it is isomorphic to
$$\prod_{\mathfrak a\in \mathfrak A} D^{n+1} \to \prod_{\mathfrak a\in \mathfrak A} S^{n+1},$$
where $\mathfrak A$ is a basis of $H_{n+1}K$.
Consequently, in the pullback diagram of morphisms in $\cat{Comod}_{C}$
$$\xymatrix{M'\ar [d]^p\ar [r]&P(H_{n+1}K,0)\otimes C\ar [d]^{q\otimes C}\\ N\ar [r]^(0.4){k^\sharp}&H_{n+1}K\otimes C,}$$
the morphism $p$ is an element of $\mathsf{Post}_{\mathsf X\otimes C}$, since $-\otimes C$ commutes with finite products and pullbacks commute with products.

Let $d$ denote the differential on $N$. Unfolding the definition of the path object construction and of the pullback, we see that $$M'=\big(N\oplus (s^{-1} H_{n+1}K\otimes C), D\big),$$
where $Dy=dy-s^{-1}k^\sharp(y)$ for all $y\in N$ and $Ds^{-1}z\otimes c=0$ for all $z\in H_{n+1}K$ and $c\in C$. For degree reasons, if $y\in N_{\leq n}$, then $Dy=dy$, and if $y\in N_{n+1}$, then $Dy=dy-s^{-1}k(y)$.  Furthermore, $Df(x)=df(x)$ for all $x\in M$, since $K$ is the cokernel of $f$.  Finally, since $C$ is $1$-connected, $M'_{n+1}=N_{n+1}$.

Let $i:M\to M'$ be the morphism of $C$-comodules induced by the pair of morphisms $f:M\to N$ and $0:M\to P(H_{n+1}K,0)$.  It is clear that $i$ is degreewise injective, since $f$ is.  Furthermore, the analysis above of the structure of $M'$ shows that $H_{*}i$ is degreewise injective and $H_{k}i$ is an isomorphism for all $k\leq n+1$, i.e., $i$ is an $(n+1)$-equivalence. Since $pi=f$, we can conclude.
\end{proof}

The proof of Theorem \ref{thm:modcat-chcomod} is now complete.

To study homotopy coinvariants of chain coalgebra coactions, we need to understand fibrant replacements in $\cat{Comod}_{C}$.  We now now show that the cobar construction actually gives rise to a fibrant replacement functor on $\cat{Comod}_{C}$.

Let $C$ be a $1$-connected chain coalgebra. Let $M$ and $N$ denote connected chain complexes, endowed with a right $C$-coaction $\rho: M\to M\otimes C$ and left $C$-coaction $\lambda :N\to C\otimes N$.  Let $\Om (M;C;N)$ denote the conormalization of the usual cosimpicial chain complex built from $\rho$, $\lambda$ and the comultiplication on $C$, i.e.,
$$\Om (M;C;N) =(M\otimes Ts^{-1} C_{+}\otimes N, D),$$
where $T$ denotes the tensor algebra functor and (modulo signs, which are given by the Koszul rule)
\begin{align*}
D(x\otimes \si c_{1}|\cdots &|\si c_{n}\otimes y)\\ =&dx\otimes \si c_{1}|\cdots|\si c_{n}\otimes y \pm x_{i}\otimes \si a^{i}|\si c_{1}|\cdots|\si c_{n}\otimes y\\
&\pm x\otimes d_{\Om}(\si c_{1}|\cdots|\si c_{n})\otimes y \\
&\pm x\otimes \si c_{1}|\cdots|\si c_{n}\otimes dy \pm x\otimes \si c_{1}|\cdots|\si c_{n}|\si b^{j}\otimes y_{j}.
\end{align*}
Here, $d$ denotes the differentials on both $M$ and $N$, and $d_{\Om}$ denotes the usual differential on the reduced cobar construction, while $\rho(x)=x_{i}\otimes a^{i}$ and $\lambda (y)=b^{j} \otimes y_{j}$.  Note that $\Om (\Bbbk;C;\Bbbk)$ is the usual reduced cobar construction, $\Om C$.

If $N=C$, then $\Om(M;C;C)$ is naturally a right $C$-comodule, via a ``cofree''  (i.e., cofree when forgetting differentials) coaction 
$$\hat\rho: \Om(M;C;C)\to \Om (M;C;C)\otimes C: x\otimes w\otimes c \mapsto x\otimes w\otimes c_{i}\otimes c^{i},$$
where $\Delta (c)=c_{i}\otimes c^{i}$.

Let $j: M \to \Om (M;C;C): x \to x_{i}\otimes 1\otimes c^{i}$, where $\rho(x)=x_{i}\otimes c^{i}$. One can show easily that $j$ is a quasi-isomorphism and a map of $C$-comodules. It is an amusing exercise to show that
$$\xymatrix@1{M\cof ^(0.4){j}_(0.4){\simeq}& \;\Om (M;C;C) \fib &\; 0}$$
is a fibrant replacement of $M$ in $\cat {Comod}_{C}$, i.e., that $\Om (M;C;C)\to 0$ is an $\mathsf X\otimes C$-Postnikov tower (Definition \ref{defn:postnikov}).  It follows that 
$$M^{hco\, C}= \Om (M;C;C)^{co\, C}\cong \Om (M;C;\Bbbk),$$ 
and therefore, from the classical definition of $\operatorname{Cotor}$, that
$$H_{*}M^{hco\, C}= \operatorname{Cotor}^C(M,\Bbbk),$$
as expected.

\subsubsection{Chain algebras}\label{subsec:ch-alg}

Let $\Bbbk$ be a field, and let $\cat M$ again be the category of finite-type chain complexes of $\Bbbk$-vector spaces. Let $\cat{Alg}$ denote the category of monoids in $\cat M$. i.e., the category of finite-type chain $\Bbbk$-algebras.

There is a model category structure on $\cat{Alg}$ in which a morphism  is a fibration if it is surjective in positive degrees, while a weak equivalence is a quasi-isomorphism.  A cofibration in $\cat{Alg}$ is a retract of the inclusion of a chain algebra $(A,d)$ as a subobject of an algebra formed by free adjunction $(A\coprod TV, D)$, where $T$ denotes the tensor algebra functor.    The usual tensor product of chain complexes induces a monoidal structure on $\cat {Alg}$, which therefore becomes a monoidal model category.

The forgetful functor from $\cat{Alg}$ to $\cat M$ is a right adjoint and therefore creates limits in $\cat {Alg}$.  In particular, the chain complex underlying a finite limit of chain algebras is isomorphic to the limit of the underlying complexes.  Consequently, $-\otimes A$ commutes with finite limits in $\cat{Alg}$, for all chain algebras $A$.

Let $H$ be a $1$-connected comonoid in $\cat {Alg}$.    Since all finite limits commute with $-\otimes H$, the category $\cat{Alg}_{H}$ of $H$-comodule algebras is finitely complete.  Furthermore, $U_{H}'$ creates colimits in $\cat{Alg}_{H}$, since $-\otimes H$ commutes with colimits,  so $\cat{Alg}_{H}$ is cocomplete as well. 

By a proof very similar to that of Theorem \ref{thm:modcat-chcomod}, we can show that the category of $H$-comodule algebras admits a  left-induced model structure.

\begin{thm} Let $H$ be a $1$-connected comonoid in $\cat {Alg}$.  The forgetful functor $U_{H}':\cat {Alg}_{H}\to \cat M$ left-induces a model structure on $\cat{Alg}_{H}$ with Postnikov presentation $\big(\mathsf {Fib}_{\cat{Alg}}\otimes H, (\mathsf {Fib}_{\cat{Alg}}\cap \mathsf {WE}_{\cat{Alg}})\otimes H\big)$.
\end{thm}

The only significant difference between the proof of this theorem and that of Theorem \ref{thm:modcat-chcomod} resides in the description of the ``based path object'' construction (cf. Proof of Lemma \ref{lem:inductive-step}), which we must apply to the $H_{n+1}K$, where $K$ is the cokernel of an injective morphism $M\to N$ of $H$-comodule algebras.   Viewing $\Bbbk \oplus H_{n+1}K$ as a algebra with trivial multplication, let $P(H_{n+1}K,0)$ be the chain algebra with trivial multiplication $(\Bbbk \oplus H_{n+1}K\oplus s^{-1}H_{n+1}K, D)$, where $D$ is defined on generators as in the chain complex case. Since the multiplication in this algebra is trivial, the differential is a derivation, as required.  Furthermore, the projection map
$$P(H_{n+1}K,0)\to \Bbbk \oplus H_{n+1}K$$
is an algebra map, since the multiplication is trivial in both the source and the target.
The ``degree reason'' arguments in the proof of Lemma \ref{lem:inductive-step} still go through, in slightly modified form, for this algebraic ``based path fibration.''

As in the chain complex case, we can show that the cobar construction actually gives rise to a fibrant replacement functor on $\cat{Alg}_{H}$.
Corollary 3.6 of \cite{hess-levi} states that if $A$ is a connected right $H$-comodule algebra, then the two-sided cobar construction $\Om (A;H;\Bbbk)$ admits a chain algebra structure that extends the obvious right $\Om H$-module structure.  Furthermore,  the quotient map $q:\Om (A;H;\Bbbk)\to A$ is a map of algebras.
Dually, if $B$ is a connected left $H$-comodule algebra, then $\Om (\Bbbk;H;B)$ admits a chain algebra structure that extends the obvious left $\Om H$-module structure.  A simple computation shows that 
$$\Om (A;H;B)\cong \Om (A;H;\Bbbk)\underset {\Om H}\otimes \Om (\Bbbk;H;B),$$
and therefore $\Om (A;H;B)$ is also naturally a chain algebra, for any right $H$-comodule algebra $A$ and left $H$-comodule algebra $B$.

It follows from the characterization of the model structure above that if $H$ is a $1$-connected chain bialgebra and $A$ is any connected right $H$-comodule algebra, then 
$$\xymatrix{ A \cof^(0.35)\sim &\Om (A;H;H) \fib & 0}$$
is a fibrant replacement of $A$ in $\cat {Alg}_{H}$, since $\Om (A;H;H)\to 0$ is a $(\mathsf {Fib}_{\cat{Alg}}\otimes H)$-Postnikov tower. As in the chain complex case, we now have
$$A^{hco\, H}=\big(\Om (A;H;H)\big)^{co\, H}\cong \Om (A;H;R).$$

\begin{rmk}  It is likely that the methods of proof applied to showing that the forgetful functor left-induces model structure on the category of comodules when the underlying category $\cat M$ is the category of chain complexes or of chain algebras can be generalized to any category $\cat M$ in which Postnikov decompositions of objects can be built inductively and in which there is a natural, decreasing sequence of successive approximations to the set of weak equivalences in $\cat M$.  Thus, for example, based on work of Mandell and Shipley \cite{mandell-shipley} and of Dugger and Shipley \cite{dugger-shipley}, it is reasonable to expect that we can define homotopy coinvariants when the underlying category is the category of symmetric spectra or of symmetric ring spectra.
\end{rmk}

\section{Homotopic Hopf-Galois extensions}

Having established a rigorous theory of homotopy coinvariants, we are ready to generalize the notion of Hopf-Galois extensions to monoidal model categories.  Once we have stated the definition of homotopic Hopf-Galois extensions, we present examples of such extensions, including trivial extensions and extensions in two topologically interesting model categories.

\begin{convention}  Throughout this section let $(\cat M, \otimes, I)$  be a monoidal model category, and let $H$ be a bimonoid in $\cat M$.  We suppose furthermore that the category $\cat{Alg}_{H}$ of $H$-comodule algebras admits a model structure with respect to which the coinvariants functor $\operatorname{Coinv}: \cat{Alg}_{H}\to  \cat{Alg}$ is a right Quillen functor, where $\cat {Alg}$ is endowed with a model structure right-induced by the forgetful functor $U_{Alg}:\cat{Alg}\to \cat M$.  Finally, any category of modules over a monoid in $M$ is considered to be endowed with the model structure right-induced by the forgetful functor.
\end{convention}

Let $B$ be any monoid in $\cat M$.  Recall that the tensor product of a right $B$-module $M$ with right action map $r$ and a left $B$-module with left action map $\ell$ is the coequalizer
$$M\otimes B\otimes N \egal {r\otimes N}{M\otimes \ell} M\otimes N \to M\underset B\otimes  N,$$
which is computed in $\cat M$.

\begin{defn}\label{defn:hg-homotopic} Let $A$ be an $H$-comodule algebra, with right $H$-coaction $\rho: A\to A\otimes H$ and multiplication map $\mu_{}:A\otimes A\to A$.  Let $B$ be a monoid in $\cat M$.

Let  $\varphi: \operatorname{Triv} (B)\to A$  be a morphism in $\cat {Alg}_{H}$. The \emph{Galois map} associated to $\varphi$ is a morphism $\beta_{\vp}: A\underset B \otimes A\to A\otimes H$ in $\cat {M}$  that is equal to the composite
$$A\underset B \otimes A \xrightarrow {A\underset B\otimes \rho}  A\underset B \otimes A\otimes H \xrightarrow{\bar\mu_{}\otimes H} A\otimes H,$$
where $\bar \mu_{}:A\underset B\otimes A \to A$ is the unique morphism from the coequalizer induced by the multiplication map of $A$. 

The map $\varphi: \operatorname{Triv} (B)\to A$ of $H$-comodule algebras is  a \emph{homotopic $H$-Hopf-Galois extension} if
\begin{enumerate}
\item the Galois map $\beta_{\vp}$ is a weak equivalence in $\cat M$, and
\item there is a choice of fibrant replacement $\xymatrix@1{j:A\wecof &A'}$ in $\cat {Alg}_{H}$ such that
$$-\underset B\otimes A^{hco\, H}: \cat {Mod}_{B}\adjunct {}{} \cat {Mod}_{A^{hco\, H}}: i_{\vp}^*$$
is a pair of Quillen equivalences, where $i_{\vp}:B\to A^{hco\, H}$ is the morphism of monoids given by the composite
$$B\cong \big(\operatorname{Triv}(B)\big)^{co\, H}\xrightarrow {\vp^{co\, H}} A^{co\, H}\xrightarrow{j^{co\, H}} (A')^{co\, H}=:A^{hco\, H}.$$
\end{enumerate}
\end{defn}

\begin{rmk} \label{rmk:beta}The Galois map is in general not a morphism of monoids, unless $\mu_{}$ is a morphism of monoids, i.e., unless $A$ is a commutative monoid.  
In section \ref{subsec:co-ring} we provide a categorical perspective on $\beta_{\vp}$, in terms of co-rings over $A$.
\end{rmk} 

\begin{rmk} Condition (2) in the definition above  replaces the object-level condition (2) of the classical definition of Hopf-Galois extensions (Definition \ref{defn:hg-classical}) with a Morita-type category-level condition in the homotopic case.  Note, however, that condition (2) implies that  if $M$ is a cofibrant $B$-module such that $M\underset B\otimes A^{hco\, H}$ is fibrant,  
then $i_{\vp}$ induces a weak equivalence
$$i_{M}:=M\underset B\otimes i_{\vp}: M\xrightarrow \sim i_{\vp }^*(M\underset B\otimes A^{hco\, H}),$$ 
which is the unit of the adjunction.  In particular, if $B$ is cofibrant as a $B$-module (e.g., if the unit $I$ of the monoidal structure is cofibrant in $\cat M$) and $A^{hco\, H}$ is fibrant in $\cat M$, then $i_{\vp }$ itself is a weak equivalence, and we recover an object-level generalization of condition (2) in Definition \ref{defn:hg-classical}.  In Remark \ref{rmk:tautology} we discuss the possbility of ``categorifying'' condition (1) as well.
\end{rmk}

In our study of homotopic Hopf-Galois extensions, we occasionally have need of the following notion.

\begin{defn}  An object $X$ in $\cat M$ is \emph{homotopically flat} if the functor $X\otimes -:\cat M \to \cat M$ preserves weak equivalences.
\end{defn}

For example, all topological spaces and all simplicial sets are homotopically flat in their respective cartesian model categories.  Moreover, the K\"unneth theorem implies that any chain complex over a field is homotopically flat.

\subsection{Trivial extensions}

 Let $R$ be a commutative ring. As explained, e.g., by Schauenburg in Example 2.1.2 of \cite {schauenburg}, an $R$-bialgebra $H$ (in the classical sense of the word) is a Hopf algebra (i.e., admits an antipode) if and only if $H$ is an $H$-Hopf-Galois extension of $R$, which is true if and only if the Galois map
$$H\otimes H \xrightarrow {H\otimes \Delta} H\otimes H\otimes H \xrightarrow {\mu\otimes H} H\otimes H$$
is  an isomorphism. More generally, a trivial extension $B\otimes \eta: B\to B\otimes H$ is $H$-Hopf-Galois if and only if $H$ is a Hopf algebra.

Motivated by this observation, we formulate the following definition.

\begin{defn} A bimonoid $H$ in a monoidal model category $(\cat M, \otimes, I)$ is a \emph{Hopf monoid} if the Galois map $\beta_{\eta}:H\otimes H\to H\otimes H$ associated to the $H$-comodule algebra map $\eta:\operatorname{Triv}(I)\to H$ is a weak equivalence in $\cat M$.
\end{defn}

\begin{exs}  If $\cat M=\cat {Top}$ and $H$ is the monoid of Moore loops on a based space $X$, then $H$ is Hopf monoid in $\cat M$.  Similarly, the chain bialgebra $C_{*}H$ of singular chains on $H$ is a Hopf monoid in $\cat {Ch}_{R}$.
\end{exs}

 We now show  that if $H$ is a Hopf monoid and $B$ is a monoid satisfying certain technical conditions, then the trivial extension $B\otimes \eta: B\to B\otimes H$ is a homotopic $H$-Hopf-Galois extension.

\begin{rmk} Observe that the following diagram commutes.
$$\xymatrix{ (B\otimes H)\underset B\otimes (B\otimes H)\ar [dd]_{\cong}\ar [rr]^{(B\otimes H)\underset B\otimes (B\otimes \Delta)}\ar[rrdd]^{\beta_{B\otimes \eta}}&&(B\otimes H)\underset B\otimes (B\otimes H)\otimes H\ar [dd]^{\mu_{B\otimes H}\otimes H}\\
\\
B\otimes H\otimes H\ar [r]^(0.45){B\otimes H\otimes \Delta}\ar @/_1.5pc/ [rr]_{B\otimes \beta_{\eta}}&B\otimes H\otimes H\otimes H\ar [r]^{B\otimes \mu\otimes H}&B\otimes H\otimes H}$$
It follows that $\beta_{B\otimes \eta}$ is a weak equivalence if $\beta_{\eta}$ is a weak equivalence and $B$ is homotopically flat.
\end{rmk}

We again separate our analysis into two parts: the cartesian case and the case of model structure on $\cat {Alg}_{H}$ left-induced by the forgetful functor $U_{H}':\cat {Alg}_{H}\to \cat {Alg}$.

\begin{prop} Let $(\cat M, \times, e)$ be a monoidal model category such that the forgetful functor right-induces a model structure on $\cat{Alg}$.  Let $H$ be  a Hopf monoid  $\cat M$, where the comultiplication is the diagonal map $\Delta:H\to H\times H$.  If $B$ is a monoid that is fibrant and homotopically flat in $\cat M$, then the trivial extension $B\times \eta: B\to B\times H$ is a homotopic $H$-Hopf-Galois extension.
\end{prop} 

\begin{proof}Since $H$ is a Hopf monoid and $B$ is homotopically flat, the Galois map $\beta_{B\times \eta}$ is a weak equivalence.

For any object $X$ of $\cat M$, the cofree comodule structure on $X\times H$ arises from the projection map $X\times H \to H$, which is a fibration in $\cat M$ if and only if $X$ is fibrant in $\cat M$.  In other words, $X\times H$ is a fibrant $H$-comodule if and only if $X$ is a fibrant object in $\cat M$.  Consequently, if $X$ is fibrant, then
$$(X\times H)^{hco\, H}=(X\times H)^{co\, H}\cong X.$$

Note that $B$ is fibrant in $\cat {Alg}$, since the model structure on $\cat{Alg}$ is right-induced by the model structure on $\cat M$, and $B$ is supposed to be fibrant in $\cat M$.  It follows that $B\times H$ is a fibrant $H$-comodule algebra and therefore that  $$i_{B\times \eta}:B\to (B\times H)^{hco\, H}\cong B$$ is an isomorphism.  Consequently, for any $B$-module $M$, the induced map $i_{M}:M\to M\underset B\times (B\times H)^{hco\, H}$ is also an isomorphism.
\end{proof}

In the case of left-induced model structure, we have the following result.

\begin{prop} Let $(\cat M, \otimes, I)$ be a monoidal category endowed  with a  model category structure.  Let  $H$ be a bimonoid in  $\cat M$.   Suppose that $\cat {Alg}$ admits a model structure right-induced by the forgetful functor $U_{Alg}:\cat {Alg}\to \cat M$ and  that $\cat {Alg}_{H}$ admits a model structure left-induced by the forgetful functor $U'_{H}:\cat {Alg}_{H}\to \cat{Alg}$.  If $B$ is a monoid that is fibrant in $ \cat{Alg}$ and homotopically flat in $\cat M$, then the trivial extension $B\to B\otimes H$ is a homotopic $H$-Hopf-Galois extension.
\end{prop}

\begin{proof}  As in the previous proof, we can conclude immediately that the Galois map $\beta_{B\otimes \eta}$ is a weak equivalence.

Let $e$ denote the terminal object in $\cat M$, which is also the terminal object in $\cat {Alg}$, since $U_{Alg}$ is a right adjoint.  Note that $e\otimes H$ is the terminal object in $\cat {Alg}_{H}$, since $-\otimes H$ is also a right adjoint.  

Since $U'_{H}$ left-induces the model structure on $\cat{Alg}_{H}$, the cofree functor $-\otimes H$ preserves fibrations. Thus, since $B\to e$ is a fibration of algebras, $B\otimes H\to e\otimes H$ is a fibration of $H$-comodule algebras, i.e., $B\otimes H$ is fibrant in $\cat {Alg}_{H}$. It follows that 
$$(B\otimes H)^{hco\, H}=(B\otimes H)^{co \, H}\cong B$$
and therefore that 
$$M\underset B\otimes i_{\vp}:M\to M\underset B \otimes (B\otimes H)^{hco\, H}$$
is an isomorphism for all $B$-modules $M$.
\end{proof}

\subsection{Examples}\label{subsec:examples}

We present in this section characterizations and explicit examples of homotopic Hopf-Galois extensions in two model categories of topological interest.

\subsubsection {Simplicial monoids}  Let  $H$ be a simplicial monoid, seen as a simplicial bimonoid, with comultiplication equal to the diagonal map.  Let $A$ be a fibrant $H$-comodule algebra, i.e., a simplicial monoid endowed with a simplicial homomorphism $p:A\to H$ that is a Kan fibration.  Let $B$ be a simplicial monoid, and let $\vp\in \cat{Alg}_{H}\big(\operatorname{Triv}(B), A\big)$.

The computations in section \ref{subsec:simplmon} imply that  if $\vp $ is homotopically $H$-Hopf-Galois, then $B$ is weakly equivalent to the fiber of $p$, i.e., $\vp$ is homotopy equivalent to a principal fibration. Furthermore, the Galois map $\beta_{\vp}$
$$A\underset B\times A \xrightarrow{A\underset B\times (A\times p)\Delta} A\underset B\times A \times H \xrightarrow {\mu_{A}\times H} A\times H$$
is a weak equivalence.

For example, suppose that $H$ is a simplicial group, $B$ is a simplicial monoid that is a Kan complex,  and $A$ is a twisted cartesian product \cite{may} of $H$ and $B$ via a twisting function $\tau :H_{\bullet}\to G_{\bullet-1}$, where $G$ is a simplicial group acting on $B$ via a map of simplicial monoids $\alpha: B\times G \to G$. We require furthermore that $\tau$ be a homomorphism in each level, so that the componentwise multiplication in $A$ is a simplicial map. 

The projection map 
$$A=B\underset \tau \times H \to H$$ is then a simplicial homomorphism and Kan fibration, i.e., $A$ is  fibrant in $\cat {Alg}_{H}$.  Moreover, the inclusion $\vp: B\hookrightarrow A$ is a homotopic $H$-Hopf-Galois extension, since the Galois map
$$A\underset B\times A \cong B\underset {\tau\times \tau}\times(H\times H)\to A\times H=(B\underset \tau\times H)\times H: (b,x,y)\mapsto (b, xy,y)$$
admits an inverse
$$(B\underset \tau\times H)\times H\to  B\underset {\tau\times \tau}\times (H\times H)\to A\times H=: (b,x,y)\mapsto (b, xy^{-1},y),$$
i.e., $\beta_{\vp}$ is actually an isomorphism.  Moreover, since $B$ itself is one model for $A^{hco\, H}$, we can take $i_{\vp}$  to be the identity morphism of $B$, thereby fulfilling condition (2) of Definition \ref{defn:hg-homotopic} trivially.

\subsubsection {Chain algebras}
Let $\Bbbk$ be a field. Let $H$ be a $1$-connected bimonoid in the category $\cat {Ch}_{\Bbbk}$ of finite-type chain complexes of $\Bbbk$-vector spaces. It is well known that any connected bimonoid in $\cat {Ch}_{\Bbbk}$ is a Hopf monoid.  Indeed, the map $\beta _{\eta}$, i.e., the composite
$$H\otimes H \xrightarrow {H\otimes \Delta} H\otimes H\otimes H \xrightarrow {\mu \otimes H} H\otimes H,$$
is  actually an isomorphism. 

Let $A$ be a connected  $H$-comodule algebra, and let $B$ be a connected chain algebra.  Let $\vp\in \cat{Alg}_{H}\big(\operatorname{Triv}(B), A\big)$.

Recall the computations done in section \ref{subsec:ch-alg}. The map $\vp$ is a homotopic $H$-Hopf-Galois extension  only if
$$i_{\vp}: B\to A^{hco \, H}= \Om (A;H;\Bbbk)$$
and
$$\beta_{\vp}:A\underset B\otimes A\to A\otimes H$$
are weak equivalences of chain algebras.

As a specific example of  a homotopic $H$-Hopf-Galois extension, we can therefore take the inclusion
$$\iota:\Om (A;H;\Bbbk) \hookrightarrow \Om (A;H;H).$$
Since 
$$\Om (A;H;H)\underset {\Om (A;H;\Bbbk)} \otimes \Om (A;H;H) \cong \Om (A;H;H)\otimes H$$
 as $H$-comodule algebras, the Galois map $\beta_{\iota}$ can be identfied with $\Om(A;H;H)\otimes \beta_{\eta}$.  It follows that both $\beta_{\iota}$ and the induced map $i_{\iota}$ are actually isomorphisms in this case.

\section{Homotopically faithful flatness and descent}

In this section we initiate a program to prove a homotopic version of an important structure theorem for Hopf-Galois extensions, due to Schneider \cite {schneider}, which relates faithful flatness and descent.  We begin by a general discussion of categories of modules endowed with coactions of co-rings, of which the category of descent data is one example.

Throughout this section, we impose the following convention.

\begin{convention}\label{conv:monmod}  Henceforth, $(\cat M, \otimes, I)$ denotes a cofibrantly generated monoidal category satisfying the monoid axiom and such that all objects are small relative to $\cat M$.  All monoidal model categories are supposed to be symmetric and closed.
\end{convention}

It follows from Theorem 4.1 in \cite{schwede-shipley} that for any monoid $A$ in $\cat M$, the category $\cat {Mod}_{A}$ admits a cofibrantly generated model structure that is right-induced by the forgetful functor $U_{A}:\cat {Mod}_{A}\to \cat M$.  In what follows we always assume that this is the model structure on $\cat {Mod}_{A}$. 

\subsection{Homotopy theory of comodules over co-rings}\label{subsec:co-ring}

Let $(A,\mu, \eta)$ be a monoid in $\cat M$, and let $_{A}\cat {Mod}_{A}$ denote the category of $A$-bimodules.  It is easy to check that $_{A}\cat {Mod}_{A}$ is a monoidal category, with monoidal product $-\underset A\otimes -$ and unit $A$, as monoidal model categories are closed monoidal and therefore the tensor product commutes with colimits on both sides.

\begin{defn}  An \emph{$A$-co-ring} is a comonoid in $(_{A}\cat {Mod}_{A},-\underset A\otimes -)$.  In other words, an $A$-co-ring is an $A$-bimodule $W$ that is endowed with a coassociative, counital comultiplication $\psi:W\to W\underset A\otimes W$ that is a morphism of $A$-bimodules.
\end{defn}

\begin{exs}
(0) The monoid $A$ is always trivially an $A$-co-ring, where the comultiplication is the isomorphism $A\to A\underset A\otimes A$ and the counit is the identity.
\smallskip 

(1) Let $(C,\Delta, \ve)$ be any comonoid in $\cat M$.  The tensor product $A\otimes C$ is then naturally an $A$-co-ring, called the \emph{trivial co-ring on $C$}. Its left $A$-module action is given by
$$A\otimes A\otimes C \xrightarrow{\mu\otimes C} A\otimes C,$$
and its right $A$-action by
$$A\otimes C\otimes A\xrightarrow\cong A\otimes A\otimes C \xrightarrow{\mu\otimes C} A\otimes C,$$
where  we have used the symmetry isomorphism $C\otimes A \cong A\otimes C$ in the second composite. Its comultiplication $\psi_{triv}$   is equal to 
$$A\otimes C\xrightarrow {A\otimes \Delta} A\otimes C\otimes C\cong (A\otimes C)\underset A\otimes (A\otimes C),$$
while the counit $\epsilon_{triv}$ is
$$A\otimes C\xrightarrow {A\otimes \ve} A\otimes I\cong A.$$
It is easy to check that both are morphisms of $A$-bimodules.
\smallskip

(2)  This example resembles example (1) superficially, but does in fact differ significantly. Let $(H,\mu_{H}, \eta, \Delta, \ve)$ be any bimonoid in $\cat M$, and let $A$ be an $H$-comodule algebra with multiplication map $\mu _{A}$ and right $H$-coaction $\rho: A \to A\otimes H$. The tensor product $A\otimes H$ is then naturally an $A$-co-ring. Its left $A$-module action is given by
$$A\otimes A\otimes H \xrightarrow{\mu_{A}\otimes H} A\otimes H,$$
and its right $A$-action by
$$A\otimes H\otimes A\xrightarrow{A\otimes H\otimes \rho} A\otimes H\otimes A\otimes H\xrightarrow {\cong} A\otimes A\otimes H\otimes H \xrightarrow {\mu_{A}\otimes \mu _{H}} A\otimes H,$$
where  we have used the symmetry isomorphism $H\otimes A \cong A\otimes H$ in the second composite. Its comultiplication $\psi_{\rho}$   is equal to 
$$A\otimes H\xrightarrow {A\otimes \Delta} A\otimes H\otimes H\cong (A\otimes H)\underset A\otimes (A\otimes H),$$
while the counit $\epsilon_{\rho}$ is
$$A\otimes H\xrightarrow {A\otimes \ve} A\otimes I\cong A.$$
It is easy to check that both are morphisms of $A$-bimodules.  Henceforth, we denote this co-ring $W^\rho$ and call it the \emph{co-ring associated to $\rho$}.
\smallskip

(3) Let $\vp:B\to A$ be any morphism of monoids in $\cat M$.  The \emph{canonical co-ring on $\vp$} has as underlying $A$-bimodule $A\underset B\otimes A$, endowed with a comultiplication $\psi_{can}$, which is equal to the composite
$$A\underset B\otimes A\cong A\underset B\otimes B\underset B\otimes A\xrightarrow{A\underset B\otimes\vp\underset B\otimes A} A\underset B\otimes A\underset B\otimes A\cong (A\underset B\otimes A) \underset A\otimes (A \underset B\otimes A).$$
As is clear from the universal property of coequalizers, the morphism $\bar \mu:A\underset B\otimes A\to A$ induced by the multiplication map of $A$ is the counit of $\psi_{can}$. 
\end{exs}

To describe the relationship between Hopf-Galois extensions and faithful flatness, we need to work with categories of the following sort.

\begin{defn}  Let $(W,\psi, \epsilon)$ be an $A$-co-ring.  The category $\cat {M}^W_{A}$ is the category of \emph{$W$-comodules in the category of right $A$-modules}.  In other words, an object of $\cat{M}^W_{A}$ is a right $A$-module $M$ together with a morphism $\theta:M\to M\underset A\otimes W$ of right $A$-modules such that  the diagrams
 $$\xymatrix{M\ar[rr]^\theta \ar [d]^\theta&&M\underset A\otimes W\ar [d]^{\theta\underset A\otimes W}&& M\ar [r]^(0.4)\theta\ar[dr]^=&M\underset A\otimes W\ar[d]^{M\underset A\otimes \epsilon}\\
M\underset A\otimes W\ar [rr]^{M\underset A\otimes \psi}&&M\underset A\otimes W\underset A\otimes W&&&M}$$
commute. Morphisms in $\cat M^W_{A}$ are morphisms of $A$-modules that respect the $W$-coactions.
\end{defn}

\begin{rmk}  The co-ring $W$ is itself an object in $\cat M^W_{A}$,  since $\psi:W\to W\underset A\otimes W$ can be viewed as a morphism of right $A$-modules.
\end{rmk}

For the specific co-rings described in the examples above, we obtain particularly interesting categories of comodules.

\begin{exs} (1) Let $C$ be a comonoid, and let $(W,\psi, \epsilon)=(A\otimes C, \psi_{triv},\epsilon_{triv})$.  Then $\cat M^W_{A}$ is isomorphic to the category of $A$-modules endowed with a $C$-coaction that is a morphism of $A$-modules, since 
$$M\underset A\otimes W=M\underset A\otimes A\otimes C\cong M\otimes C$$ 
for any right $A$-module $M$.  Under this isomorphism $M\otimes C$ is endowed with the right $A$-action given by the composite 
$$M\otimes C\otimes A\cong M\otimes A\otimes C \xrightarrow {r\otimes C} M\otimes C,$$
where $r$ is the right $A$-action on $M$.
\smallskip

(2) Let $H$ be a bimonoid and $A$ an $H$-comodule algebra with right $H$-coaction $\rho$. Let $(W,\psi, \epsilon)=(A\otimes H, \psi_{\rho},\epsilon_{\rho})$.  Then $\cat M^W_{A}$ is isomorphic to the category of $A$-modules endowed with an $H$-coaction that is a morphism of $A$-modules, since 
$$M\underset A\otimes W=M\underset A\otimes A\otimes H\cong M\otimes H$$ 
for any right $A$-module $M$.  Under this isomorphism $M\otimes H$ is endowed with the right $A$-action given by the composite
$$M\otimes H\otimes A\xrightarrow {M\otimes H\otimes \rho} M\otimes H\otimes A\otimes H \cong M\otimes A\otimes H\otimes H\xrightarrow {r\otimes \mu_{H}} M\otimes H,$$
where $r$ is the right $A$-action on $M$.  Note that $(A,\rho)$ itself is an object in $\cat M^W_{A}$.
\smallskip

(3) Let $\vp: B\to A$ be any morphism of monoids in $\cat M$, and let $$W_{\vp}=(A\underset B\otimes A, \psi_{can},\epsilon_{can}),$$ the canonical co-ring associated to $\vp$.  The category $\cat M^{W_{\vp}}_{A}$ is isomorphic to $\cat D(\vp)$, the \emph{descent category associated to $\vp$}. An object of $\cat D(\vp)$  is a right $A$-module $M$ endowed with a morphism $\theta: M\to M\underset B\otimes A$ such that the diagrams
$$\xymatrix{M\ar[rr]^\theta \ar [d]^\theta&&M\underset B\otimes A\ar [d]^{\theta\underset B\otimes A}&&M\ar [r]^(0.4)\theta\ar[dr]^=&M\underset B\otimes A\ar[d]^{\bar r}\\
M\underset B\otimes A\ar [rr]^{M\underset B\otimes \vp \underset B\otimes A}&&M\underset B\otimes A\underset B\otimes A&&&M}$$
commute, where $\bar r$ is induced by the right $A$-action on $M$.  We refer to $(M,\theta)$ as a \emph{descent datum}.
The morphisms in $\cat D(\vp)$ are $A$-module morphisms respecting the structure maps.

The key to showing that $\cat M^{W_{\vp}}_{A}$ and $\cat D(\vp)$ are isomorphic is the observation that 
$$M\underset A\otimes W=M\underset A\otimes A\underset B\otimes A\cong M\underset B\otimes A$$
for all right $A$-modules $M$.
\end{exs}

\begin{rmk} Any morphism $\gamma:W\to W'$ of $A$-co-rings induces a functor
$$\gamma_{*}:\cat M^W_{A}\to \cat M^{W'}_{A},$$
which is defined on objects by $\gamma_{*}(M, \theta)=\big(M,(M\underset A\otimes \gamma)\circ\theta\big)$.   If equalizers exist in $\cat M^W_{A}$, then $\gamma_{*}$ admits a right adjoint
$$-\underset {W'}\square \gamma_{*}(W): \cat M^{W'}_{A}\to \cat M^W_{A},$$
where for any object $(M',\theta')$ in $\cat M^{W'}_{A}$, the diagram
$$(M',\theta')\underset {W'} \square\gamma_{*}(W)\to \operatorname{equal}(M'\underset A\otimes W\egal{\theta'\underset A\otimes W}{M\underset A'\otimes (\gamma\underset A\otimes W)\psi} M'\underset A\otimes W'\underset A\otimes W)$$
is an equalizer, computed in $\cat M^W_{A}$.  

To prove that $-\underset {W'}\square \gamma_{*}(W)$ truly is the right adjoint to $\gamma_{*}$, note that a morphism in $\cat M^W_{A}$ from $(M,\theta)$ to $(M', \theta')\underset {W'}\square \gamma_{*}(W)$ is equivalent to a morphism
$$f: (M,\theta)\to (M'\underset A\otimes W, M'\underset A\otimes \psi)$$
in $\cat M^W_{A}$ such that
$$(\theta'\underset A\otimes W)f= \big(M'\underset A\otimes (\gamma\underset A\otimes W)\psi \big)f.$$
Straightforward diagram chases then show that the composite
$$M\xrightarrow f M'\underset A\otimes W \xrightarrow{M\underset A\otimes \epsilon} M'\underset A\otimes A \cong M'$$
is a morphism in $\cat M^{W'}_{A}$.
\end{rmk}

\begin{ex} Let $H$ be a bimonoid in $\cat M$, and let $A$ be an $H$-comodule algebra with $H$-coaction map $\rho$.  Let $B$ be a monoid in $\cat M$.
Let  $\varphi: \operatorname{Triv} (B)\to A$  be a morphism in $\cat {Alg}_{H}$.   The Galois map
$$\beta_{\vp}:A\underset B\otimes A \to A\otimes H$$
underlies a morphism of $A$-co-rings, from the canonical co-ring $W_{\vp}$ associated to $\vp$ to the co-ring $W^\rho$ associated to $\rho$.  When checking that 
$$(\beta _{\vp}\underset A\otimes \beta_{\vp})\circ \psi_{can}=\psi_{\rho}\circ \beta_{\vp},$$
it is very important to remember that the right $A$-action on $A\otimes H$ is defined using the coaction $\rho$.

The Galois map therefore induces a functor
$$(\beta_{\vp})_{*}: \cat D(\vp)\to \cat M^{W_\rho}_A,$$
which we call the \emph{Galois functor associated to $\vp$}.
\end{ex}

To discuss descent theory in a homotopical context, we need a model category structure on  $\cat M^W_{A}$, for certain co-rings $W$.  The next lemma, which is easy to prove, is the first step towards obtaining the desired structure.

\begin{lem} Let $(W,\psi, \epsilon)$ be an $A$-co-ring.  The forgetful functor $U_{W}:\cat M^W_{A}\to \cat {Mod}_{A}$ admits a left adjoint $-\underset A\otimes W: \cat {Mod}_{A}\to \cat M^W_{A}$, where the $W$-coaction on $M\underset A\otimes W$ is defined to be 
$$M\underset A\otimes \psi: M\underset A\otimes W\to M\underset A\otimes W\underset A\otimes W.$$
\end{lem}

We can now apply the machinery of section \ref{subsec:induced}, in particular Corollary \ref{cor:postnikov},  to deducing the existence of model category structure on $\cat M^W_{A}$.

\begin{thm}\label{thm:co-ring} Assuming Convention \ref{conv:monmod}, let $A$ be a monoid in $\cat M$.  Let $W$ be an $A$-co-ring  such that $\cat M^W_{A}$ is finitely bicomplete, and let $U_{W}:\cat M^W_{A}\to \cat {Mod}_{A}$ denote the forgetful functor. Let 
$$\mathsf W=U_{C}^{-1}(\mathsf{WE}_{\cat {Mod}_{A}}) \text { and } \mathsf C=U_{C}^{-1}(\mathsf{Cof}_{\cat {Mod}_{A}}),$$
and
$$\mathsf X=  \mathsf{Fib}_{\cat {Mod}_{A}} \underset A\otimes W\text { and }\mathsf Z= (\mathsf{Fib}_{\cat {Mod}_{A}}\cap\mathsf{WE}_{\cat {Mod}_{A}}) \underset A\otimes W.$$
If $\mathsf{Post}_{Z}\subset \mathsf W$ and for all $f\in \mor \cat M^W_{A}$
there exist 
\begin{enumerate}
\item [(a)] $i\in \mathsf{C} $ and $p\in \mathsf{Post}_{ \mathsf Z}$ such
that $f=pi$; 
\smallskip
\item [(b)]$j\in \mathsf{C}\cap \mathsf{W}$ and $q\in \mathsf{Post}_{\mathsf X}$ such that
$f=qj$,
\end{enumerate}
then    $\mathsf W$, $\mathsf C$  and $\widehat{\mathsf{Post}_{\mathsf X}}$ are the weak equivalences, cofibrations and fibrations in a model category structure on $\cat M^W_{A}$, with respect to which 
$$U_{W}:\cat M^W_{A} \adjunct {}{}Ê\cat {Mod}_{A}: -\underset A\otimes W$$ 
is a Quillen pair.
\end{thm}

\begin{rmk}  If $\gamma:W\to W'$ is a morphism of $A$-co-rings such that both $\cat M^W_{A}$ and $\cat M^{W'}_{A}$ admit model structures left-induced by forgetting comodule structure, then it is easy to see that  the induced functor $\gamma_{*}:\cat M^W_{A}\to \cat M^{W'}_{A}$ preserves both cofibrations and weak equivalences.  It follows that 
$$\gamma_{*}:\cat M^W_{A}\adjunct{}{} \cat M^{W'}_{A}:-\underset {W'}\square \gamma_{*}(W)$$
is a Quillen pair, which is a Quillen equivalence if for all cofibrant objects $(M,\theta)$ in $\cat M^W_{A}$ and all fibrant objects $(M',\theta')$ in $\cat M^{W'}_{A}$, a morphism in $\cat M^W_{A}$
$$f:(M,\theta)\to (M',\theta')\underset {W'}\square \gamma_{*}(W)$$
is a weak equivalence if and only if its transpose
$$f^\flat :\gamma_{*}(M,\theta)\to (M',\theta')$$
is a weak equivalence.  In particular, if $(M,\theta)$ is cofibrant in $\cat M^W_{A}$ and $\gamma_{*}(M,\theta)$ is fibrant in $\cat M^{W'}_{A}$, then the unit of the adjunction
$$\eta_{M}:(M,\theta)\to \gamma_{*}(M,\theta)\underset {W'}\square \gamma_{*}(W)$$
must be a weak equivalence in $\cat M^W_{A}$, if $\gamma_{*}$ is a Quillen equivalence.

Recall that, because of our choice of model structure on $\cat M^W_{A}$, an object $(M, \theta)$ is cofibrant in $\cat M^W_{A}$ if and only if the underlying $A$-module $M$ is cofibrant in $\cat {Mod}_{A}$ and that a morphism in $\cat M^W_{A}$ is a weak equivalence if and only if the underlying morphism in $\cat M$ is a weak equivalence.
\end{rmk}

\subsection{The structure theorem}

Schneider's structure theorem relates $H$-Hopf-Gal\-ois extensions of rings $\vp:B\to A$ and the category $\cat M^{W_\rho}_A$. Before stating the theorem, we need to introduce yet another pair of adjoint functors. Recall that, if $H$ is a bimonoid and $A$ is an $H$-comodule algebra, we can view the objects in $\cat M^{W_\rho}_A$ as $A$-modules $M$ equipped with an $H$-coaction $\theta: M\to M\otimes H$ that is a morphism of $A$-modules. This is the point of view adopted in the definition below.

\begin{defn} Let $H$ be a bimonoid in $\cat M$, and let $A$ be an $H$-comodule algebra with $H$-coaction map $\rho$.  Let $B$ be a monoid in $\cat M$. Let  $\varphi: \operatorname{Triv} (B)\to A$  be a morphism in $\cat {Alg}_{H}$, and let $W_{\rho}$ denote the co-ring $(A\otimes H, \psi_{\rho}, \epsilon _{\rho})$.

The \emph{$\rho$-induction functor} 
$$\operatorname{Ind}_{\rho}: \cat{Mod}_{A^{co\, H}}\to \cat M^{W_{\rho}}_{A}$$  is defined on objects by 
$$\operatorname{Ind}_{\rho} (M)=(M\underset {A^{co\, H}}\otimes A, M\underset {A^{co\, H}}\otimes \rho ),$$
while the \emph{$\rho$-coinvariants functor} 
$$(-)^{co \, \rho}:\cat M^{W_{\rho}}_{A}\to \cat{Mod}_{A^{co\, H}},$$ 
is defined on objects so that for all $(M,\theta)$,
the diagram
$$(M, \theta)^{co \, \rho}\to \operatorname{equal}(M\egal{\theta}{M\otimes \eta} M\otimes H)$$
is an equalizer, computed in $\cat {Mod}_{A^{co\, H}}$.
\end{defn}

\begin{rmk} It is not difficult to show that $(\operatorname{Ind}_{\rho}, (-)^{co\, \rho})$ is an adjoint pair.
\end{rmk}

We can now state Schneider's structure theorem.

\begin{thm} \cite{schneider} Let $\Bbbk$ be a commutative ring, and let $H$ be a $\Bbbk$-flat Hopf algebra.  The following are equivalent for any $H$-comodule algebra $A$, with coinvariant algebra $B=A^{co\, H}$.
\begin{enumerate}
\item The inclusion $B\hookrightarrow A$ is an $H$-Hopf-Galois extension, and $A$ is a faithfully flat $B$-module.
\item The functor $\operatorname{Ind}_{\rho}: \cat {Mod}_{B}\to \cat M_{A}^{W_{\rho}}$ is an equivalence, where $\rho$ denotes the $H$-coaction on $A$.
\end{enumerate}
\end{thm}

In \cite{schauenburg} Schauenburg provides an elegant proof of Schneider's theorem, based on the characterization of faithfully flat ring extensions in terms of descent.  Our goal is to construct a homotopic version of Schauenburg's argument, in order to prove an analogue of  Schneider's theorem.

Constructing an argument like Schauenburg's requires that we specify what we mean by faithful flatness of monoid extensions in model categories.  We begin by recalling how faithfully flat descent works for rings.

\begin{defn} The \emph{canonical descent datum} functor $$\operatorname{Can}: \cat{Mod}_{B}\to \cat D(\vp),$$ is defined on objects by $\operatorname{Can}(M)=(M\underset B\otimes A, \theta_{M})$, with $\theta_{M}=M\underset B\otimes \vp \underset B\otimes A$.  The functor $\operatorname{Can}$ admits a right adjoint $$\operatorname{Coinv}:\cat D(\vp)\to \cat {Mod}_{B},$$ where
$$\operatorname{Coinv}(N,\theta)=\operatorname{equal}( N\egal {\theta}{N\underset B\otimes \vp} N\underset B\otimes A).$$
\end{defn}

We can now formulate faithfully flat descent for rings, for which one reference is \cite{benabou-roubaud}.  The formulation we choose is based on Theorem 4.5.2 in \cite{schauenburg}.

\begin{thm} \label{thm:faithflat} Let $\vp: B\to A$ be an inclusion of rings. The functor $\operatorname{Can}: \cat{Mod}_{B}\to \cat D(\vp)$ is an equivalence of categories, with inverse $\operatorname{Coinv}:\cat D(\vp)\to \cat {Mod}_{B}$, if and only if $A$ is faithfully flat as a $B$-module.
\end{thm}

The definition of homotopic faithful flatness proposed here is inspired by this theorem.   We begin by showing that the adjoint pairs introduced in this section are Quillen pairs, under appropriate hypotheses.

\begin{convention} \label{conv:co-ring} Henceforth we suppose that  the canonical co-ring $W_{\vp}$ associated to $\vp$ and the co-ring $W_{\rho}$ associated to $\rho$ are such that the forgetful functors to  $\cat {Mod}_{A}$ left-induce model category structure on $M^{W}_{A}$ for $W=W_{\vp}, W_{\rho}$ and therefore on $\cat D(\vp)$. For example, if the hypotheses of Theorem \ref{thm:co-ring} are satisfied, then this convention holds. 
\end{convention}

\begin{lem}  Assuming Conventions \ref{conv:monmod} and  \ref{conv:co-ring}, let $H$ be a bimonoid in $\cat M$, and let $A$ be an $H$-comodule algebra with $H$-coaction map $\rho$.  Let $B$ be a monoid in $\cat M$. Let  $\varphi: \operatorname{Triv} (B)\to A$  be a morphism in $\cat {Alg}_{H}$, and let $W_{\rho}$ denote the co-ring $(A\otimes H, \psi_{\rho}, \epsilon _{\rho})$.

The adjoint pairs
$$\operatorname{Ind}_{\rho}: \cat{Mod}_{A^{co\, H}}\to \cat M^{W_{\rho}}_{A}:(-)^{co\, \rho}$$
and  
$$\operatorname{Can}:\cat {Mod}_{B}\adjunct{}{}\cat D(\vp):\operatorname{Coinv}$$ 
are Quillen pairs.
\end{lem}

\begin{proof} We do the proof for the pair $(\operatorname{Can}, \operatorname{Coinv})$; the case of the other pair is essentially identical. Let $\op I$ and $\op J$ denote the sets of generating cofibrations and of generating acyclic cofibrations of $\cat M$, respectively.  The sets of generating cofibrations and the generating acyclic cofibrations of $\cat {Mod}_{B}$ are then $\op I\otimes B$ and $\op J\otimes B$, while those of $\cat {Mod}_{A}$ are $\op I\otimes A$ and $\op J\otimes A$.  Recall that in a cofibrantly generated model category, any (acyclic) cofibration is  a retract of the composition of a directed system 
$$M_{0}\to M_{1}\to \cdots \to M_{\beta}\to M_{\beta +1}\to \cdots$$
where $M_{\beta+1}$ is obtained from $M_{\beta}$ by pushing out along a generating (acyclic) cofibration.  

Since the model structure on $\cat D(\vp)$ is left-induced by the forgetful functor to $\cat {Mod}_{A}$, if $i\otimes B$ is an (acyclic) generating cofibration in $\cat {Mod}_{B}$, then 
$$\operatorname{Can}(i\otimes B)=(i\otimes B)\underset B\otimes A=i\otimes A$$ 
is an (acyclic) cofibration in $\cat D(\vp)$.  Recall that the set of cofibrations in a model category is closed under pushouts, direct limits and retracts. Consequently, the image of any (acyclic) cofibration under the functor $\operatorname{Can}$ is an (acyclic) cofibration, as $\operatorname{Can} $ preserves both colimits and retracts.
\end{proof}

We can now formulate a homotopic version of faithful flatness, motivated by Theorem \ref{thm:faithflat}.

\begin{defn} Let $\vp:B\to A$ be a morphism of monoids in $\cat M$.   The monoid $A$ is \emph{homotopically faithfully flat} over $B$ if $$\operatorname{Can}:\cat {Mod}_{B}\adjunct{}{}\cat D(\vp):\operatorname{Coinv}$$ is a Quillen equivalence.
\end{defn}

In other words, $A$ is homotopically faithfully flat over $B$ if for any cofibrant $B$-module $M$ and fibrant descent datum $(N,\theta)$, a morphism of $B$-modules $$f: M\to \operatorname{Coinv} (N,\theta)=\operatorname{equal}( N\egal {\theta}{N\underset B\otimes \vp} N\underset B\otimes A)$$ is a weak equivalence if and only if its transpose $$f^\flat: (M\underset B\otimes A, M\underset B\otimes \vp\underset B\otimes A) \to (N, \theta)$$ is a weak equivalence of descent data.  In particular, if $M$ is a cofibrant $B$-module and $(M\underset B\otimes A, M\underset B\otimes \vp\underset B\otimes A)$ is a fibrant descent datum, then the unit of the adjunction
$$\eta_{M}:M\to \operatorname{Coinv} (M\underset B\otimes A, M\underset B\otimes \vp\underset B\otimes A)$$
must be a weak equivalence if $A$ is homotopically faithfully flat over $B$.

We conjecture that the following analogue of Lemma 2.3.5 in \cite {schauenburg} should hold, at least under strong enough conditions on $\vp$. 

\begin{conj} \label{conj:descent} Assuming Conventions \ref{conv:monmod} and \ref{conv:co-ring}, let $H$ be a bimonoid in $\cat M$. Suppose that the category $\cat{Alg}_{H}$ of $H$-comodule algebras admits a model structure with respect to which the coinvariants functor $\operatorname{Coinv}: \cat{Alg}_{H}\to  \cat{Alg}$ is a right Quillen functor. 

 Let $A$ be an $H$-comodule algebra, and let  $\varphi: \operatorname{Triv} (B)\to A$  be a morphism in $\cat {Alg}_{H}$, with associated Galois map $\beta_{\vp}:A\underset B\otimes A\to A\otimes H$.
 
 If $\vp$ is a homotopic $H$-Hopf-Galois extension, then 
 $$(\beta_{\vp})_{*}: \cat D(\vp)\to \cat M_{A}^{W_{\rho}}:-\underset {A\otimes H} \square(\beta_{\vp})_{*}(A\underset B\otimes A)$$
is  a pair of  Quillen equivalencs.
 \end{conj}
 
 \begin{rmk} If $\beta_{\vp}: A\underset B\otimes A\to A\otimes H$ is actually an isomorphism, then it follows from the proof of Lemma 2.3.5 in \cite{schauenburg} that $(\beta_{\vp})_{*}$ is an equivalence of categories.  The conjecture therefore holds for those homotopic Hopf-Galois extensions, like both of those treated in section \ref{subsec:examples},  such that $\beta_{\vp}$ is an isomorphism.
 \end{rmk}
 
 \begin{rmk}\label{rmk:tautology} It might be appropriate to render this conjecture a tautology, by replacing condition (1) in  the definition of homotopic Hopf-Galois extensions (Definition \ref{defn:hg-homotopic}) by the following condition.
 \begin{quote} ($1'$) The Galois functor
 $$(\beta_{\vp})_{*}: \cat D(\vp)\to \cat M_{A}^{W_{\rho}}$$
is a Quillen equivalence.\end{quote}
This modification would certainly be in the spirit of condition (2) in Definition \ref{defn:hg-homotopic}, which is also a category-level, rather than object-level, description.  Further experience with explict Hopf-Galois extensions should make it clear which is actually the ``correct'' definition of homotopic Hopf-Galois extensions.
\end{rmk}
 
 \begin{rmk} To prove this conjecture for arbitrary $\vp$, if we choose not to render it a tautology, it may be necessary to weaken slightly the definition of a descent datum and to work with ``homotopic descent data,'' rather than strict descent data.
 \end{rmk}
 
 We can now formulate and prove a homotopic version of Schneider's theorem, at least under the hypothesis that the conjecture above is true.
 
Recall the adjunction 
$$-\underset B\otimes A^{hco\, H}: \cat {Mod}_{B}\adjunct {}{} \cat {Mod}_{A^{hco\, H}}: i_{\vp}^*$$
from Definition \ref{defn:hg-homotopic}, which is a pair of Quillen equivalences if $\vp$ is a homotopic Hopf-Galois extension.

 \begin{thm} Assume Conventions \ref{conv:monmod} and \ref{conv:co-ring}. Let $H$ be a bimonoid in $\cat M$ such that the category $\cat{Alg}_{H}$ of $H$-comodule algebras admits a model structure with respect to which the coinvariants functor $\operatorname{Coinv}: \cat{Alg}_{H}\to  \cat{Alg}$ is a right Quillen functor. 

 Let $A$ be a $H$-comodule algebra with fibrant underlying object in $\cat M$.  Let $B$ be a monoid such that the functor $M\underset B\otimes -:_{B}\negthinspace\cat {Mod}\to \cat M$ commutes with equalizers up to weak equivalence.  Let $\varphi: \operatorname{Triv} (B)\to A$  be a morphism in $\cat {Alg}_{H}$ such that $A$ is cofibrant as a $B$-module and $M\underset B\otimes A$ is fibrant in $\cat M^{W_\rho}_A$ for all cofibrant $B$-modules $M$. 
 
 If Conjecture \ref{conj:descent} holds, then the following conditions are equivalent.
 \begin{enumerate}
 \item The monoid map $\vp$ is a homotopic $H$-Hopf-Galois extension, and $A$ is homotopically faithfully flat over $B$.
 \item The functor $ \operatorname{Ind}_{\rho}\circ (-\underset B\otimes A^{co\, H}):\cat {Mod}_{B}\to \cat M_{A}^{W_\rho}$ is a Quillen equivalence.
 \end{enumerate}
 \end{thm}

\begin{proof} Our proof is inspired by the proof of Corollary 2.3.6 in \cite {schauenburg}.  We begin by observing that the following diagram of functors clearly commutes.
\begin{equation}\label{eqn:diagram} \xymatrix{\cat D(\vp)\ar [rr]^{(\beta_{\vp})_{*}}\ar[dr]_{\operatorname{Coinv}}&&\cat M^{W_\rho}_A\ar [dl]^{i_{\vp}^*\circ (-)^{co\,\rho}}\\ &\cat {Mod}_{B}}\end{equation}

$(1)\Rightarrow (2):$   Conjecture \ref{conj:descent} implies that $(\beta_{\vp})_{*}$ is a Quillen equivalence.  On the other hand, by definition of homotopic faithful flatness, $\operatorname{Coinv}$ is a Quillen equivalence.  We conclude from  diagram (\ref{eqn:diagram}) that $i_{\vp}^*\circ  (-)^{co\,\rho}$  is a Quillen equivalence, which implies that its left adjoint,   $\operatorname{Ind}_{\rho}\circ (-\underset B\otimes A^{co\, H})$, is also a Quillen equivalence.

$(2)\Rightarrow (1):$ The hypothesis that $-\underset B\otimes A$ is a Quillen equivalence implies that the unit of the adjunction
$$\eta_{M}: M\to i_{\vp}^*(M\underset B\otimes A, M\underset B\otimes \rho)^{co\,\rho}$$
is a weak equivalence in $\cat {Mod}_{B}$ for all cofibrant $B$-modules $M$, since  $M\underset B\otimes A$ is fibrant in $\cat M^{W_\rho}_A$ by hypothesis. In particular, since  $A$ is supposed to be cofibrant as a $B$-module, 
$$\eta_{A}: A\xrightarrow \sim i_{\vp}^*(A\underset B\otimes A, A\underset B\otimes \rho)^{co\,\rho}$$
is a weak equivalence.  

Observe that since $\beta_{\vp}$ is a morphism of co-rings, it can also be viewed as a morphism in $\cat M^{W_{\rho}}_{A}$.  Moreover,  the following diagram commutes, thanks to the universal property of the equalizer.
$$\xymatrix{A\ar[r]^(0.35){\eta_{A}}\ar [dr]_{=}&i_{\vp}^*(A\underset B\otimes A, A\underset B\otimes \rho)^{co\,\rho}\ar [d]^{i_{\vp}^*(\beta_{\vp})^{co\,\rho}}\\
&i_{\vp}^*(A\otimes H, A\otimes \Delta)^{co\,\rho}\cong A}$$
Thus, $i_{\vp}^*(\beta_{\vp})^{co\,\rho}$ is a weak equivalence.  Moreover, the fibrancy hypothesis on $A$ implies that $\mathbb R(i_{\vp}^*\circ (-)^{co\,\rho})(\beta_{\vp})=i_{\vp}^*(\beta_{\vp})^{co\,\rho}$, whence $\beta_{\vp}$ must also be a weak equivalence, since $\mathbb R(i_{\vp}^{*}\circ(-)^{co\,\rho})$ is an equivalence of categories. 

To conclude that $\vp$ is a homotopic Hopf-Galois extension, observe that for any cofibrant $B$-module $M$, the unit map $\eta_{M}$ is equal to the composite
$$M\xrightarrow{i_{M}} i_{\vp}^*(M\underset B\otimes A^{co\, H})\xrightarrow {i_{\vp}^*u_{M}} i_{\vp}^*(M\underset B\otimes A, M\underset B\otimes \rho)^{co\,\rho},$$
where $u_{M}:M\underset B\otimes A^{co\, H} \to (M\underset B\otimes A, M\underset B\otimes \rho)^{co\,\rho}$ is the morphism of $A^{co \, H}$-modules  induced by the natural map $u:A^{co \, H}\to A$.  Since $M\underset B\otimes -$ commutes with equalizers up to weak equivalence, $u_{M}$ and therefore also $i_{\vp }^*u_{M}$ are weak equivalences.  Consequently, $i_{M}$ is a weak equivalence for all cofibrant $M$ and therefore $i_{\vp }^*$ is a Quillen equivalence.

Since $\vp $ is  a homotopic Hopf-Galois extension, it follows from Conjecture \ref{conj:descent} that $(\beta_{\vp})_{*}$ is a Quillen equivalence.  The commuting diagram (\ref{eqn:diagram}) then implies that $\operatorname{Coinv}$ is a Quillen equivalence, i.e., that $\vp$ is homotopically faithfully flat.
\end{proof}

\begin{rmk} Note that it follows from the proof above that, under the various cofibrancy and fibrancy conditions, $\vp $ is a homotopic Hopf-Galois extension whenever $-\underset B\otimes A$ is a  Quillen equivalence, without any need of Conjecture \ref{conj:descent}, which serves only to make the connection with homotopical faithful flatness.
\end{rmk}

\section{Appendix: model categories and derived functors}

\subsection{Definitions and terminology}

We recall here certain elements of the theory of model categories, primarily to fix notation and terminology.

\begin{defn}   A \emph{model category} consists of a 
category $\cat M$, together with classes of morphisms $\mathsf{WE}, \mathsf{Fib}, \mathsf{Cof}\subset\mor 
\cat M$ that are closed under composition and contain all identities,
such that the following axioms are satisfied.  
\begin{enumerate} 
\item [(M1)]
All finite limits and colimits in $\cat M$ exist. 
\smallskip
 
\item [(M2)] Let $\xymatrix@1{f:A\ar [r]&B}$ and
$\xymatrix@1{g:B\ar [r]&C}$ be morphisms in $\cat M$.  If two of $f$, $g$, and $gf$ are
in $\mathsf{WE}$, then so is the third. 
\smallskip
  
\item [(M3)] The classes $\mathsf{WE}$, $\mathsf{Fib}$, and $\mathsf{Cof}$
 are all closed under taking retracts.  
\smallskip
 
\item [(M4)] Given a commuting diagram in $\cat M$
$$\xymatrix{ A\ar[r]^f\ar [d]^{i} & E\ar[d]^p\\ X\ar[r]^g& B,}$$
there is a morphism $h:X\to E$ such that $ph =g$ and $ hi=f$ if
\begin{enumerate}
\item [(a)] $i\in \mathsf{Cof}$ and  $p\in \mathsf{Fib}\cap \mathsf{WE}$, or
\item [(b)] $i\in\mathsf{Cof} \cap \mathsf{WE}$ and $p\in \mathsf{Fib}$.  
\end{enumerate}
\smallskip
 
\item [(M5)] If $f\in \mor \cat M$,
then there exist 
\begin{enumerate}
\item [(a)] $i\in \mathsf{Cof} $ and $p\in \mathsf{Fib}\cap \mathsf{WE}$ such
that $f=pi$; 
\item [(b)]$j\in \mathsf{Cof}\cap \mathsf{WE}$ and $ q\in \mathsf{Fib}$ such that
$f=qj$. 
\end{enumerate} 
\end{enumerate} 

The \emph{homotopy category} of a model category $\cat M$, denoted $\operatorname{Ho} \cat M$, is the localization of $\cat M$ with respect to $\mathsf{WE}$.
\end{defn}

By analogy with the homotopy structure in the category of topological 
spaces, the morphisms belonging to the classes $\mathsf{WE}$, $\mathsf{Fib}$ and $\mathsf{Cof}$ are 
called \emph{weak equivalences}, \emph{fibrations}, and \emph{ 
cofibrations} and are denoted by decorated 
arrows $\xymatrix@1{{}\we&}$, $\xymatrix@1{{}\fib&}$, and
 $\xymatrix@1{{}\cof&}$.  The elements of the classes $\mathsf{Fib}\cap \mathsf{WE}$ and
$\mathsf{Cof}\cap \mathsf{WE}$ are called, respectively, \emph{ acyclic fibrations} and
\emph{ acyclic cofibrations}.  Since $\mathsf{WE}$, $\mathsf{Fib}$ and $\mathsf{Cof}$ are all
closed under composition and contain all isomorphisms, we can and
sometimes do view them as subcategories of $\cat M$, rather than
simply as classes of morphisms.

Axiom (M1) implies that any model category has an initial object $\phi$ and a 
terminal object $e$.  An object $A$ in a model category is \emph{ 
cofibrant} if the unique morphism $\xymatrix@1{\phi \ar[r]&A}$ is a cofibration.  
Similarly, $A$ is \emph{ fibrant} if the unique morphism $\xymatrix@1{A\ar [r]&e}$ is a 
fibration.

When definining homotopy coinvariants, we need the following notion.

\begin {defn}  Let $\cat M$ and $\cat M'$ be model categories.  A pair of adjoint functors $F:\cat M\adjunct {}{} \cat M':G$ is a \emph{Quillen pair} if $F$ preserves cofibrations and $G$ preserves fibrations.  
\end{defn}

\begin{rmk} As is well known \cite {hovey}, $(F,G)$ is a Quillen pair if and only if $F$ preserves both cofibrations and acyclic cofibrations, which is true if and only if $G$ preserves fibrations and acyclic fibrations.
\end{rmk}  

\begin{prop}  A Quillen pair $F:\cat M\adjunct {}{} \cat M':G$ induces a pair of adjoint functors 
$$ \mathbb L F:\operatorname{Ho} \cat M \adjunct {}{} \operatorname{Ho} \cat M':\mathbb RG.$$
\end{prop}

\begin{rmk} For any objects $X$ of $\cat M$ and $X'$ of $\cat M'$,  $\mathbb LF(X)$ is represented by $F(QX)$ and $\mathbb RG(X')$ by $G(RX')$, where $\xymatrix@1{\emptyset\cof &QX\wefib &X}$ is a cofibrant replacement of $X$ and $\xymatrix@1{X'\wecof &RX'\fib &e}$ is a fibrant replacement of $X'$.
\end{rmk}

\begin {defn}  Let $\cat M$ and $\cat M'$ be model categories.  A Quillen pair 
$$F:\cat M\adjunct {}{} \cat M':G$$ 
is a \emph{Quillen equivalence} if for every cofibrant object $X$ in $\cat M$ and every fibrant object $X'$ in $\cat M'$, a morphism $f:X\to GX'$ is a weak equivalence in $\cat M$ if and only if its transpose $f^\flat :FX\to X'$ is a weak equivalence in $\cat M'$.  It follows that $(\mathbb LF, \mathbb RG)$ is an equivalence of categories.  
\end{defn}

\subsection{Induced model structures}\label{subsec:induced}

A common way of creating model structures is by transfer across adjunctions.  We need in this article to use both right-to-left and left-to-right transfer, as specified in the following definition.

\begin{defn}\label{defn:induction} Let $G: \cat C \to \cat M$ be a functor, where $\cat M$ is a model category. A model structure on $\cat C$ is \emph{right-induced} from $\cat M$ if $\mathsf {WE}_{\cat C}=G^{-1}(\mathsf {WE}_{\cat M})$ and $\mathsf {Fib}_{\cat C}=G^{-1}(\mathsf {Fib}_{\cat M})$.

Let $F: \cat C \to \cat M$ be a functor, where $\cat M$ is a model category. A model structure on $\cat C$ is \emph{left-induced} from $\cat M$ if $\mathsf {WE}_{\cat C}=F^{-1}(\mathsf {WE}_{\cat M})$ and $\mathsf {Cof}_{\cat C}=F^{-1}(\mathsf {Cof}_{\cat M})$.
\end{defn}

\begin{rmk} In general, functors to model categories do not induce model structures on their sources.   If, however, $\cat M$ is cofibrantly generated, and $G:\cat C\to \cat M$ admits a left adjoint $F$, then there are well-known conditions on $F$ and $G$ and their relation to the generating (acyclic) cofibrations in $\cat M$ that ensure the existence of a right-induced model structure on $\cat C$ (cf., e.g., Theorem 11.3.2 in \cite{hirschhorn}).  Left induction is less well understood, probably because fibrantly generated model categories are rare.
\end{rmk}

The next theorem is key to determining conditions under which left-induced model structures exist.  Before stating the theorem, we introduce a bit of useful notation.

\begin{notn}Let $\mathsf X$ be any subset of morphisms in a category $\cat C$. 
\begin{enumerate}
\item The closure of $\mathsf X$ under formation of retracts is denoted $\widehat{\mathsf X}$, i.e., 
$$f\in \widehat{\mathsf X}\Longleftrightarrow  \exists\; g\in \mathsf X \text{ such that $f$ is a retract of $g$} .$$
\item The set of morphisms with the right lifting property with respect to $\mathsf X$ is denoted $\mathsf{RLP(X)}$.  In other words,  a morphism $p:E\to B$  is in $\mathsf {RLP(X)}$ if for any commuting diagram in $\cat C$
$$\xymatrix{ A\ar[r]^f\ar [d]^{i} & E\ar[d]^p\\ X\ar[r]^g& B,}$$
where $i\in \mathsf{X}$, there is a morphism $h:X\to E$ such that $ph =g$ and $ hi=f$.
\end{enumerate}
\end{notn}

\begin{rmk}\label{rmk:retract}  Note that if $\mathsf Y\subset \mathsf {RLP(X)}$, then $\widehat{\mathsf Y}\subset \mathsf {RLP(X)}$.  Furthermore, $\mathsf {RLP(X)}$ is clearly closed under pullback and inverse limits.  Finally, recall that in any model category $\mathsf{Fib}=\mathsf{RLP(Cof\cap WE)}$ and $\mathsf{Fib\cap WE}=\mathsf{RLP(Cof)}$.
\end{rmk}

\begin{thm}\label{thm:left-ind}  Let $\cat C$ be a finitely bicomplete category,  and let $\mathsf{W}, \mathsf {C}, \mathsf{P}, \mathsf{Q}$ be subsets of morphisms in $\cat C$ that are closed under composition, contain all identities and satisfy the following conditions.
\begin{enumerate}
\item Let $\xymatrix@1{f:A\ar [r]&\,B}$ and
$\xymatrix@1{g:B\ar [r]&\,C}$ be morphisms in $\cat M$.  If two of $f$, $g$, and $gf$ are
in $\mathsf{W}$, then so is the third. 
\smallskip
\item $\widehat {\mathsf W}=\mathsf W$ and $\widehat {\mathsf C}=\mathsf C$.
\smallskip
\item 
\begin{enumerate}
\item [(a)] $\mathsf P\subset \mathsf{RLP(C)}$. 
\item [(b)] $\mathsf Q\subset \mathsf{RLP(C\cap W)}$.
\end{enumerate}
\smallskip
\item  If $f\in \mor \cat C$,
then there exist 
\begin{enumerate}
\item [(a)] $i\in \mathsf{C} $ and $p\in \mathsf{P}$ such
that $f=pi$; 
\item [(b)]$j\in \mathsf{C}\cap \mathsf{W}$ and $q\in \mathsf{Q}$ such that
$f=qj$. 
\end{enumerate} 
\item $\mathsf P\subset \mathsf W$.
\end{enumerate}
Then $\mathsf W$, $\mathsf C$  and $\widehat{\mathsf Q}$ are the weak equivalences, cofibrations and fibrations in a model category structure on $\cat C$.
\end{thm}

\begin{proof}  Axioms (M1), (M2), (M3) and (M5)(b) are satisfied simply by hypothesis if $\mathsf W$, $\mathsf C$  and $\widehat{\mathsf Q}$ are the weak equivalences, cofibrations and fibrations we consider.  Axiom (M4)(b) follows easily from hypothesis (3)(b), by Remark \ref{rmk:retract}.

To conclude, we show that 
$$\widehat{\mathsf Q}\cap \mathsf W=\widehat {\mathsf P},$$
which, together with (3)(a) above, implies (M4)(a)  and, together with (4)(a) above, implies (M5)(a).  

Let $(g:E\to B)\in \mathsf{RLP(C\cap W)}$.  By hypothesis (4)(b), there exist $(j:E\to E')\in \mathsf {C\cap W}$  and  $(q:E'\to B)\in \mathsf Q$ such that $g=qj$. There is thus a commutative diagram
$$\xymatrix{ E\ar[r]^=\ar [d]^{j} & E\ar[d]^g\\ E'\ar[r]^q& B}$$
where $j\in \mathsf {C\cap W}$ and $g\in \mathsf{RLP(C\cap W)}$.  It follows that there exists $r:E'\to E$ such that $gr=q$ and $rj=Id_{E}$, and therefore the diagram
$$\xymatrix{E\ar [d]_{g}\ar [r]^j&E'\ar [d]^q\ar[r]^r&E\ar [d]^g\\ B\ar [r]^=&B\ar [r]^=&B}$$
commutes, i.e., $g$ is a retract of $q$.  

We have thus established that $ \mathsf{RLP(C\cap W)}\subset \widehat{\mathsf Q}$, which, together with hypothesis (3)(b) and Remark \ref{rmk:retract}, implies that 
$$\mathsf{RLP(C\cap W)}=\widehat{\mathsf Q}.$$
A similar argument, applying hypotheses (4)(a) and (3)(a) and Remark \ref{rmk:retract}, shows that 
$$\mathsf{RLP(C)}= \widehat{\mathsf P}$$
and therefore that 
$$\widehat{\mathsf P}\subset \widehat{\mathsf Q}.$$
Conditions (2) and (5) then imply that
$$\widehat{\mathsf P}\subset \widehat{\mathsf Q} \cap \mathsf W.$$

Let $(q:E\xrightarrow \sim B)\in \widehat{\mathsf Q} \cap \mathsf W$.  By hypothesis (4)(a), there exist $(i:E\to B')\in \mathsf {C}$  and  $(p:B'\to B)\in \mathsf P$ such that $q=pi$.  By hypotheses (1) and (5), $i\in \mathsf W$.
There is thus a commutative diagram
$$\xymatrix{ E\ar[r]^=\ar [d]^{i} & E\ar[d]^q\\ B'\ar[r]^p& B}$$
where $i\in \mathsf {C\cap W}$ and $q\in \widehat{\mathsf Q} \cap \mathsf W \subset \mathsf{RLP(C\cap W)}$.  It follows that there exists $r:B'\to E$ such that $qr=p$ and $ri=Id_{E}$, and therefore the diagram
$$\xymatrix{E\ar [d]_{q}\ar [r]^i&B'\ar [d]^p\ar[r]^r&E\ar [d]^q\\ B\ar [r]^=&B\ar [r]^=&B}$$
commutes, i.e., $q$ is a retract of $p$.   Thus, $ \widehat{\mathsf Q} \cap \mathsf W\subset \widehat{\mathsf P}$, and we can conclude.
\end{proof}

We apply Theorem \ref{thm:left-ind} to proving the existence of left-induced model structures, where the set of fibrations is generated by the particular type of morphism described in the definition below.

\begin{defn} \label{defn:postnikov} Let $\mathsf X$ be a set of morphisms in a category $\cat C$ that is closed under pullbacks.  If $\lambda$ is an ordinal and  $Y:\lambda ^{op}\to \cat C$ is a functor such that for all $\beta <\lambda$, the morphism $Y_{\beta +1}\to Y_{\beta}$ fits into a pullback
$$\xymatrix{Y_{\beta +1}\ar [d]_{}\ar [r]^{}&X_{\beta +1}\ar [d]^{x_{\beta+1}}\\ Y_{\beta}\ar [r]^{k_{\beta}}&X_{\beta}}$$
for some $x_{\beta +1}: X_{\beta +1}\to X_{\beta}$ in $\mathsf X$ and  $k_{\beta}:Y_{\beta}\to X_{\beta}$ in $\cat C$, while
$Y_{\gamma}:=\lim _{\beta<\gamma}Y_{\beta}$ for all limit ordinals $\gamma<\lambda$,
then the composition of the tower $Y$
$$\lim_{\lambda^{op}}Y_{\beta}\to Y_{0},$$
\emph{if it exists}, is an \emph{$\mathsf X$-Postnikov tower}.  The set of all $\mathsf X$-Postnikov towers is denoted $\mathsf {Post}_{\mathsf X}$.

A \emph{Postnikov presentation} of a model category $\cat M$ is  a pair of  sets of morphisms $\mathsf X$ and $\mathsf Z$ satisfying 
$$\mathsf {Fib}_{\cat M}=\widehat{\mathsf {Post}_{\mathsf X}}\quad\text { and }\quad \mathsf {Fib}_{\cat M}\cap \mathsf {WE}_{\cat M}=\widehat{\mathsf {Post}_{\mathsf Z}}$$ 
and such that for all $f\in \mor \cat M$,
there exist 
\begin{enumerate}
\item [(a)] $i\in \mathsf{Cof} $ and $p\in \mathsf {Post}_{\mathsf Z}$ such
that $f=pi$; 
\item [(b)] $j\in \mathsf{Cof}\cap \mathsf{WE}$ and $ q\in \mathsf {Post}_{\mathsf X}$ such that
$f=qj$.
\end{enumerate}
\end{defn}

\begin{rmk} For any $\mathsf X$, the set  $\mathsf{Post}_{\mathsf X}$ is closed under pullbacks, since inverse limits commute with pullbacks.  Furthermore, $\mathsf{Post}_{\mathsf X}$ is clearly closed under composition as well.
\end{rmk}

\begin{rmk}\label{rmk:postnikov-lift} Let $\mathsf X$ and $\mathsf Y$ be two classes of morphisms in  a category $\cat C$ admitting pullbacks and inverse limits. It is a straightforward exercise to show that if $\mathsf X\subset \mathsf{RLP(Y)}$, then $\mathsf {Post}_{\mathsf X}\subset \mathsf{RLP(Y)}$ as well, and therefore $\widehat{\mathsf {Post}_{\mathsf X}}\subset \mathsf{RLP(Y)}$. In particular, $(\mathsf {Fib}_{\cat M}, \mathsf {Fib}_{\cat M}\cap \mathsf {WE}_{\cat M})$ is always a Postnikov presentation of a model category $\cat M$.
\end{rmk}

\begin{cor} \label{cor:postnikov} Let $\cat M$ be a model category with Postnikov presentation $(\mathsf X, \mathsf Z)$.
Let $\cat C$ be a finitely bicomplete category, and let $F:\cat C\adjunct {}{}\cat M:G$ be an adjoint pair of functors.  Let 
$$\mathsf W=F^{-1}(\mathsf{WE}_{\cat M}) \text { and } \mathsf C=F^{-1}(\mathsf{Cof}_{\cat M}).$$   
If $\mathsf{Post}_{G(\mathsf Z)}\subset \mathsf W$ and for all $f\in \mor \cat C$
there exist 
\begin{enumerate}
\item [(a)] $i\in \mathsf{C} $ and $p\in \mathsf{Post}_{G( \mathsf Z)}$ such
that $f=pi$; 
\smallskip
\item [(b)]$j\in \mathsf{C}\cap \mathsf{W}$ and $q\in \mathsf{Post}_{G(\mathsf X)}$ such that
$f=qj$,
\end{enumerate}
then    $\mathsf W$, $\mathsf C$  and $\widehat{\mathsf{Post}_{G(\mathsf X)}}$ are the weak equivalences, cofibrations and fibrations in a model category structure on $\cat C$, with respect to which $F:\cat C \adjunct {}{}Ê\cat M: G$ is a Quillen pair.
\end{cor}

Note that $(G(\mathsf X), G(\mathsf Z))$ is a Postnikov presentation of the left-induced model structure on $\cat C$.

\begin{proof} To obtain a left-induced model structure on $\cat C$, we need to show that hypotheses (1)-(5) of Theorem \ref{thm:left-ind} are satisfied, where $\mathsf P=\mathsf{Post}_{G(\mathsf Z)}$ and $\mathsf Q=\mathsf{Post}_{G(\mathsf X)}$.  Note that hypotheses (4)(a) and (b) are exactly the hypotheses (a) and (b) of this corollary, while hypothesis (5) is a hypothesis of this corollary as well.

Since $\mathsf{WE_{\cat M}}$ satisfies axiom (M2) for model categories, it is clear that $\mathsf W$ satisfies hypothesis (1) of Theorem \ref{thm:left-ind}.  Moreover, axiom (M3) for $\mathsf{WE_{\cat M}}$ and $\mathsf{Cof_{\cat M}}$ can easily be seen to imply hypothesis (2) of Theorem \ref{thm:left-ind}, as functors preserve retracts.

To prove (3)(a), consider first  a commuting diagram in $\cat C$ 
$$\xymatrix{ A\ar[r]^f\ar [d]^{i} & GE\ar[d]^{Gp}\\ X\ar[r]^g& GB,}$$
where $i\in \mathsf C$ and $p\in \mathsf Z$, which gives rises, via the adjunction between $F$ and $G$ to a commuting diagram in $\cat M$
$$\xymatrix{ FA\ar[r]^{f^\flat}\ar [d]^{Fi} & E\ar[d]^{p}\\ FX\ar[r]^{g^\flat}& B.}$$
Since $Fi\in \mathsf{Cof}_{\cat M}$ and $p\in \mathsf{Fib}_{\cat M}\cap \mathsf{WE}_{\cat M}$, axiom (M4)(a) implies that there is a morphism $h:FX\to E$ such that $p\circ h=g^\flat$ and $h\circ Fi= f^\flat$.  Applying the adjunction between $F$ and $G$ again, we obtain a commutative diagram
$$\xymatrix{ A\ar[r]^f\ar [d]^{i} & GE\ar[d]^{Gp}\\ X\ar[r]^g\ar [ur]^{h^\sharp}& GB}$$
 and thus
 $$G(\mathsf Z)\subset \mathsf{RLP(C)},$$
which implies by Remark \ref{rmk:postnikov-lift} that
 $$\mathsf{Post}_{G(\mathsf Z)} \subset \mathsf{RLP(C)}.$$
 Similarly,
  $$\mathsf{Post}_{G(\mathsf X)} \subset \mathsf{RLP(C\cap W)},$$
 i.e.,  hypothesis (3)(b) is satisfied as well.
 
 To see that the adjoint pair $(F,G)$ is a Quillen pair with respect to the newly defined model structure on $\cat C$, observe first that since $G$ is a right adjoint, it preserves limits.  Thus, the inclusion $G(\mathsf Z)\subset \mathsf{RLP(C)}$ implies, in conjunction with Remark \ref{rmk:retract}, that $G(\mathsf {Post}_{\mathsf Z})\subset \mathsf{RLP(C)}$.  Since $\mathsf{RLP(C)}$ is closed under taking retracts and $G$ preserves retracts, it follows that
 $$G(\mathsf{Fib}_{\cat M}\cap \mathsf{WE}_{\cat M})=G(\widehat{\mathsf {Post}_{\mathsf Z}})\subset \mathsf{RLP(C)}=\mathsf{Fib}_{\cat C}\cap \mathsf{WE}_{\cat C}.$$
 Similarly,
  $$G(\mathsf{Fib}_{\cat M})\subset \mathsf{RLP(C\cap W)}= \mathsf{Fib}_{\cat M}.$$ 
  \end{proof}
  
  \begin{rmk}  Let $\cat M$ be a model category with Postnikov presentation $(\mathsf X, \mathsf Z)$.
Let $\cat C$ be a bicomplete category, and let $F:\cat C\adjunct {}{}\cat M:G$ be an adjoint pair of functors.  Let 
$$\mathsf W=F^{-1}(\mathsf{WE}_{\cat M}) \text { and } \mathsf C=F^{-1}(\mathsf{Cof}_{\cat M}).$$

One can impose additional, reasonable conditions on the adjunction $(F,G)$ that guarantee that $\mathsf{Post}_{G(\mathsf Z)}\subset \mathsf W$.  For example, if 
\begin{enumerate}
\item for all $p: E\to B$ in $\mathsf Z$ and for all $g:Y\to GB$ in $\cat C$, 
$$F(Y\underset {GB}\times GE \xrightarrow {\bar p}Y)\in \mathsf{WE}_{\cat M},$$
where $\bar p$ is the induced morphism from the pullback of $Gp$ and $g$ to $Y$,
\item $F$ preserves inverse limits, and
\item the composition of a tower of weak equivalences in $\cat M$ is a weak equivalence, 
\end{enumerate}
then it is an easy exercise to show that $\mathsf{Post}_{G(\mathsf Z)}\subset \mathsf W$.
\end{rmk}
 \bibliographystyle{amsplain}
\bibliography{hhg}

\providecommand{\bysame}{\leavevmode\hbox to3em{\hrulefill}\thinspace}
\providecommand{\MR}{\relax\ifhmode\unskip\space\fi MR }
\providecommand{\MRhref}[2]{%
  \href{http://www.ams.org/mathscinet-getitem?mr=#1}{#2}
}
\providecommand{\href}[2]{#2}
\begin{thebibliography}{10}

\bibitem{benabou-roubaud}
J.~B{\'e}nabou and J.~Roubaud, \emph{Monades et descente}, C.R. Acad. Sci.
  Paris, S{\'e}rie A. \textbf{270} (1970), 96--98.

\bibitem{bujard}
C{\'e}dric Bujard, \emph{Towards a general theory of homotopic {H}opf-{G}alois
  extensions}, 2006, p.~128pp.

\bibitem{dugger-shipley}
Daniel Dugger and Brooke Shipley, \emph{Postnikov extensions of ring spectra},
  Algebr. Geom. Topol. \textbf{6} (2006), 1785--1829 (electronic).
  \MR{MR2263050 (2007g:55007)}

\bibitem{fht}
Yves F{\'e}lix, Stephen Halperin, and Jean-Claude Thomas, \emph{Rational
  homotopy theory}, Graduate Texts in Mathematics, vol. 205, Springer-Verlag,
  New York, 2001. \MR{MR1802847 (2002d:55014)}

\bibitem{hess-levi}
Kathryn Hess and Ran Levi, \emph{An algebraic model for the loop space homology
  of a homotopy fiber}, Alg. Geom. Top. \textbf{7} (2007), 1699--1765.

\bibitem{hirschhorn}
Philip~S. Hirschhorn, \emph{Model categories and their localizations},
  Mathematical Surveys and Monographs, no.~99, 2003.

\bibitem{hovey}
Mark~A. Hovey, \emph{Model categories}, Mathematical Surveys and Monographs,
  vol.~63, 1998, MR 99h:55031.

\bibitem{mandell-shipley}
M.~A. Mandell and B.~Shipley, \emph{A telescope comparison lemma for {THH}},
  Topology Appl. \textbf{117} (2002), no.~2, 161--174. \MR{MR1875908
  (2002k:55019)}

\bibitem{may}
J.~Peter May, \emph{Simplicial objects in algebraic topology}, Van Nostrand
  Mathematical Studies, No. 11, D. Van Nostrand Co., Inc., Princeton,
  N.J.-Toronto, Ont.-London, 1967. \MR{MR0222892 (36 \#5942)}

\bibitem{montgomery}
Susan Montgomery, \emph{{H}opf {G}alois theory: a survey},  (2009).

\bibitem{rognes}
John Rognes, \emph{Galois extensions of structured ring spectra}, Memoirs of
  the American Mathematical Society, vol. 192, American Mathematical Society,
  2008.

\bibitem{schauenburg}
Peter Schauenburg, \emph{{H}opf {G}alois and bi-{G}alois extensions}, Fields
  Insitute Communications \textbf{43} (2004), 469--515.

\bibitem{schneider}
H.-J. Schneider, \emph{Principal homogeneous spaces for arbitrary {H}opf
  algebras}, Israel J. Math. \textbf{72} (1990), 196--231.

\bibitem{schwede-shipley}
Stefan Schwede and Brooke~E. Shipley, \emph{Algebras and modules in monoidal
  model categories}, Proc. London Math. Soc. \textbf{80} (2000), no.~3,
  491--511.

\bibitem{weibel}
Charles~A. Weibel, \emph{An introduction to homological algebra}, Cambridge
  Studies in Advanced Mathematics, vol.~38, Cambridge University Press,
  Cambridge, 1994. \MR{MR1269324 (95f:18001)}

\end{thebibliography}

\end{document}